\documentclass[12pt,a4paper]{article}
\usepackage[latin1]{inputenc}
\usepackage{amsmath, amsthm, amssymb, amsfonts, amscd}
\usepackage[english]{babel}
\usepackage[all]{xy}
\usepackage[linktocpage=true]{hyperref}
\usepackage{color}
\usepackage[table]{xcolor}
\usepackage{pinlabel}
\usepackage{tikz-cd}
 \usepackage{booktabs}
\usepackage{enumitem}
\usepackage{geometry}
\usepackage{mathtools}
\usepackage{authblk}
\usepackage{thmtools}
\usepackage{stmaryrd}
\usepackage{comment}

\usepackage{mmacells} 
\newlength\dunder
\declaretheoremstyle[bodyfont=\normalfont,spaceabove=\medskipamount,
    spacebelow=\medskipamount]{definition}

\theoremstyle{definition}

\newtheorem{theorem}{Theorem}[section]
\newtheorem{lemma}[theorem]{Lemma}
\newtheorem{corollary}[theorem]{Corollary}
\newtheorem{proposition}[theorem]{Proposition}
\newtheorem{definition}[theorem]{Definition}

\newtheorem{remark}[theorem]{Remark}

\newtheorem{example}[theorem]{Example}

\newcommand{\defeq}{\vcentcolon=}

\title{\large \textbf{NEW QUANTUM INVARIANTS OF PLANAR KNOTOIDS}}

\author{\normalsize WOUT MOLTMAKER \& ROLAND VAN DER VEEN}

\date{}

\begin{document}

\maketitle

\begin{abstract}
In this paper we discuss the applications of knotoids to modelling knots in open curves and produce new knotoid invariants. We show how invariants of knotoids generally give rise to well-behaved measures of how much an open curve is knotted. We define biframed planar knotoids, and construct new invariants of these objects that can be computed in polynomial time. As an application of these invariants we improve the classification of planar knotoids with up to five crossings by distinguishing several pairs of prime knotoids that were conjectured to be distinct by Goundaroulis et al.
\end{abstract}


\section{Introduction}

Knotoids were introduced by V.~Turaev in \cite{turaev2012knotoids}. Intuitively, a knotoid on a surface $\Sigma$ is a knot diagram on $\Sigma$ with two open ends that are allowed to lie in the interior of the diagram. This last property is in contrast with long knots or tangles. For applications one is usually interested in the cases $\Sigma=S^2$ or $\Sigma=\mathbb{R}^2$, and in this paper we shall focus on these cases. Knotoids on $S^2$ and $\mathbb{R}^2$ are referred to as \textbf{spherical} and \textbf{planar}, respectively.

The most prevalent application of knotoids is to the topology of knotted open curves \cite{goundaroulis2017topological,dorier2018knoto,gugumcu2017new}. Examples of such curves are plentiful in chemistry and molecular biology, see for example \cite{mansfield1994there,everaers1996topological,cox2008hydrocarbon,morse2020dna}, but quantifying the knottedness of such a curve remains difficult. The problem is that finding the knottedness of an open curve is not a well-posed problem: while closed curves (knots) fall into different equivalence classes under ambient isotopy (their knot type), all open curves are clearly equivalent under ambient isotopy. In this paper we discuss how the knottedness of open curves can be rigorously quantified, and show that knotoids are a useful tool for doing so. In particular we formalize what it means to quantify the knottedness of an open curve by defining \textbf{knot measures}. We then give a large class of examples of such knot measures by showing that every invariant of spherical knotoids produces a knot measure, generalizing the results of \cite{panagiotou2020}.

After discussing how invariants of knotoids can be used to extract topological information of open curves, we go on to construct several new invariants of planar knotoids. In particular we define the universal quantum invariant of planar knotoids associated to a given ribbon Hopf algebra. In fact this is an invariant of `biframed' knotoids, which we introduce and discuss at some length in a preliminary section. We show that all quantum invariants defined in \cite{moltmaker} can be recovered from these universal quantum invariants, thereby demonstrating their `universality'. We then consider the specific example of the Hopf algebra $\mathbb{D}$ introduced in \cite{bar2021}, and discuss a computer implementation to compute portions of this invariant in polynomial time.

While spherical knotoids are relatively well-understood, being tabulated up to 6 crossings \cite{goundaroulis2019systematic}, planar knotoids have proven to be more difficult to classify. In \cite{goundaroulis2019systematic} a table of prime planar knotoids with up to 5 crossings is produced and it is found that there are between 944 and 950 prime planar knotoids with 5 crossings. (Compare this with the spherical case, where there are only 24 prime knotoids with 5 crossings.) We see that the classification of planar knotoids with 5 crossings is almost complete, with the exception of 6 unresolved pairs of knotoid diagrams that no known invariant of planar knotoids can distinguish but that are nevertheless thought to be in-equivalent. Having defined universal quantum invariants of planar knotoids we consider one such invariant, and show that it can be used to resolve 2 of the 6 unresolved pairs of planar knotoids. We thereby improve on the classification of planar knotoids, and hence on known techniques of distinguishing knotoids.\\

The outline of this paper is as follows: In section \ref{sec:definitions} we define the objects of interest to us, namely knotoids, simple theta-curves, $H$-curves, and biframed planar knotoids. Next, in section \ref{sec:measures} we discuss a general framework for how knotoids and knotoid invariants can be applied to quantify the knottedness of open curves. In section \ref{sec:invariants} we define universal quantum invariants of biframed planar knotoids, relate them to the invariants discussed in \cite{moltmaker}, and define a new quantum invariant of planar knotoids which can be computed in polynomial time based on \cite{bar2021}. In section \ref{sec:examples} we carry out example computations of the new invariant considered in section \ref{sec:invariants}, distinguishing some of the unresolved pairs of planar knotoids from \cite{goundaroulis2019systematic}.

\section{Definitions}\label{sec:definitions}

In this section we will review the basic objects of study in this paper, namely planar and spherical knotoids, biframed knotoids, and the objects in $S^3$ whose ambient isotopy classes they are in bijection with.

\subsection{Knotoids and Geometric Realizations}

\begin{definition}
\cite{turaev2012knotoids} Let $\Sigma$ be a surface. A \textbf{knotoid diagram} on $\Sigma$ is a smooth immersion $\phi:[0,1]\to \Sigma$ whose only singularities are transversal double points with over/undercrossing data. For a knotoid diagram $\phi$ we refer to $\phi(0)$ and $\phi(1)$ as the \textbf{leg} and \textbf{head} of $\phi$, respectively. We say that two knotoid diagrams are \textbf{equivalent} if they can be related by a sequence of ambient isotopies and applications of the Reidemeister moves $R1,R2,R3$ familiar for knot diagrams. A \textbf{knotoid} on $\Sigma$ is an equivalence class of knotoid diagrams on $\Sigma$.
\end{definition}

\begin{remark}
We explicitly note that the definition of equivalence of knotoid diagrams does not allow for either of the \textbf{forbidden moves} depicted in Figure \ref{fig:forbiddenmoves} that involve a crossing and an end-point. Indeed, the forbidden moves can clearly be used to render any knotoid trivial.
\begin{figure}[ht]
    \centering
    \includegraphics[width=.6\linewidth]{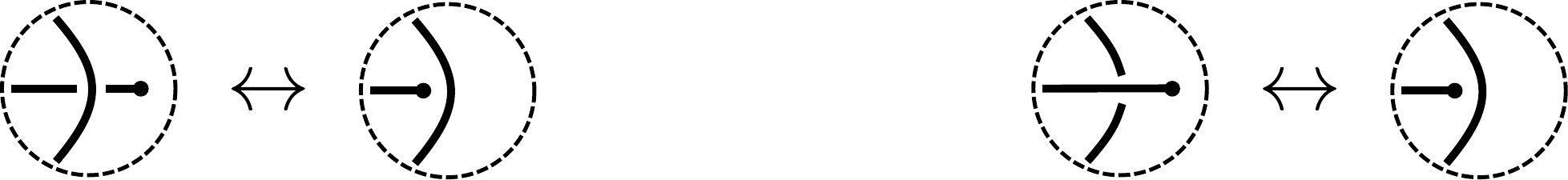}
    \caption{The forbidden moves on knotoid diagrams.}
    \label{fig:forbiddenmoves}
\end{figure}
\end{remark}

We stress that knotoids are \textit{diagrams} on $\Sigma$, and while they are essentially open-ended knot diagrams they do not correspond to open curves in some 3-manifold in the same way that classical knots correspond to closed curves in $S^3$. Instead, the class of three-dimensional objects that knotoids on $\Sigma$ are in one-to-one correspondence with depends on $\Sigma$, as we shall see below. In what follows we focus on $\Sigma\in\{\mathbb{R}^2,S^2\}$:

\begin{definition}
Knotoids on $\mathbb{R}^2$ are called \textbf{planar knotoids}, while those on $S^2$ are referred to as \textbf{spherical knotoids}.
\end{definition}



We now move on to the description of \textbf{geometric realizations} of knotoids, by which we mean sets of objects in three dimensions whose ambient isotopy equivalence classes are in in one-to-one correspondence with knotoids. For example, by Reidemeister's theorem equivalence classes of knot diagrams are in one-to-one correspondence with ambient isotopy classes of knots in $S^3$. We therefore say that knots form the geometric realization of knot diagrams. In knot theory we generally make no distinction and refer to both objects as knots. With knotoids one has to be slightly more careful, since it turns out that knotoid diagrams are not simply projections of their geometric realizations as is the case for knots. Indeed, if this were so then the geometric realization of knotoids would be open curves in $S^3$, but all open curves in $S^3$ are trivial up to ambient isotopy.

Instead, V.~Turaev showed that spherical knotoids are in one-to-one correspondence with \textbf{simple theta-curves} \cite{turaev2012knotoids}, casting these as the geometric realization of spherical knotoids:

\begin{definition}\label{def:thetacurve}
A \textbf{theta-curve} is an embedding $\theta:\Theta\to S^3$ into $S^3$ of the graph $\Theta$ that consists of two vertices $\{v_0,v_1\}$ and three edges $\{e_+,e_0,e_-\}$ between them. We say two theta-curves are \textbf{equivalent} if they can be related by an ambient isotopy of $S^3$ that preserves the labels of the vertices and edges of $\theta$. We say a theta-curve is \textbf{simple} if $\theta(e_+\cup e_-)$ is equivalent to the unknot in $S^3$.
\end{definition}

\begin{remark}
The embeddings $\Theta\hookrightarrow S^3$ defining theta-curves are always assumed to be smooth, in the sense that they are smooth at every point of $\Theta$ that has a neighbourhood homeomorphic to $\mathbb{R}$. More generally in the following we assume all embeddings of spaces $X$ into $S^3$ are smooth at all points where $X$ is locally flat.
\end{remark}

A simple theta-curve $\theta$ is essentially a knotted open curve, namely $\theta(e_0)$, whose end-points are anchored to an unknotted circle in $S^3$. This observation forms the intuition for the following result:

\begin{proposition}
\cite[Thm.~6.2]{turaev2012knotoids} Equivalence classes of simple theta-curves are in one-to-one bijection with spherical knotoids, i.e.~simple theta-curves form the geometric realtization of spherical knotoids.
\end{proposition}

In \cite{gugumcu2017new}, the results from \cite{turaev2012knotoids} on simple theta-curves were adapted to give a geometric realization of planar knotoids. We adjust their result slightly, working in the thickened plane $\mathbb{R}^2\times [-1,1]$ rather than in $\mathbb{R}^3$, and cast this geometric realization in terms of what we shall call \textbf{simple $H$-curves}:

\begin{definition}\label{def:H-curve}
Let $H$ denote the graph with vertex set $\{v_0,v_1,t_0,t_1,b_0,b_1\}$ and edge set $\{(t_i,v_i),(v_i,b_i)\}_{i=0,1}\cup\{(v_0,v_1)\}$; see Figure \ref{fig:H_graph}.
\begin{figure}[ht]
    \centering
    \includegraphics[width=.23\linewidth]{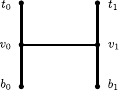}
    \caption{A drawing of the graph $H$.}
    \label{fig:H_graph}
\end{figure}

An \textbf{$H$-curve} is an embedding $\phi:H\to \mathbb{R}^2\times [-1,1]$ such that $\phi(t_i)\in\mathbb{R}^2\times\{1\}$ and $\phi(b_i)\in\mathbb{R}^2\times\{-1\}$ for $i\in\{0,1\}$. Two $H$-curves are said to be \textbf{equivalent} if they can be related by a label-preserving ambient isotopy. We refer to the lines $\phi\left( (t_i,v_i)\cup(v_i,b_i) \right)\subseteq \mathbb{R}^2\times[-1,1]$ for $i\in\{0,1\}$ as the \textbf{auxiliary lines} of the $H$-curve $\phi$.

We say an $H$-curve is \textbf{simple} if both of the $(1,1)$-tangles defined by its auxiliary lines are trivial.
\end{definition}

\begin{proposition}\label{prop:planar_realization}
Simple $H$-curves form the geometric realization of planar knotoids.
\end{proposition}
\begin{proof}
(Sketch): The bijection used to prove the proposition is as follows: given a simple $H$-curve $\phi$, apply an ambient isotopy of $\mathbb{R}^2\times [-1,1]$ that turns the auxiliary lines into straight lines $\{p\}\times[-1,1]$ and $\{q\}\times[-1,1]$ for some $p,q\in\mathbb{R}^2$. Then project the open curve $\phi\left( (v_0,v_1) \right)$ onto $\mathbb{R}^2\times\{0\}$ while recording over/under-crossing data to obtain a knotoid diagram (making some small deformations in case the projection happens to contain triple points or cusps). Conversely, given a planar knotoid diagram $\phi:I\to \mathbb{R}^2$ we turn it into a simple $H$-curve by drawing it on the plane $\mathbb{R}^2\times\{0\}$ inside $\mathbb{R}^2\times [-1,1]$, turning the diagram into an smooth embedding of $I$ by removing its double points using deformations in small neighbourhoods of the double points as specified by the over/under-crossing data, and attaching the auxiliary lines $l=\{\varphi(0)\}\times [-1,1]$ and $h=\{\varphi(1)\}\times [-1,1]$ to the leg and head respectively.

This process is depicted in Figure \ref{fig:planar_bijection}. The proof that this indeed describes a well-defined bijection is given in \cite{gugumcu2017new}.
\begin{figure}[ht]
    \centering
    \includegraphics[width=.7\linewidth]{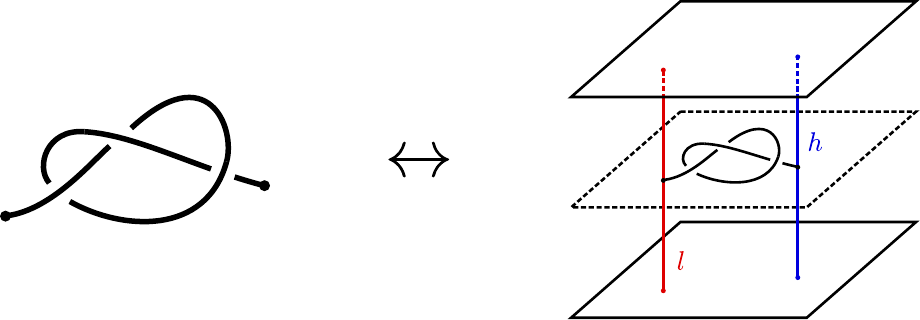}
    \caption{The bijection used in the proof of Proposition \ref{prop:planar_realization}.}
    \label{fig:planar_bijection}
\end{figure}
\end{proof}

\subsection{Biframed Knotoids}

As stated in the introduction, the classification of planar knotoids with 5 crossings is complete with the exception of 6 unresolved pairs \cite{goundaroulis2019systematic}. In section \ref{sec:invariants} we will construct a new planar knotoid invariant that can distinguish some of these pairs. However this will turn out to be an invariant of `biframed' knotoids, which were defined in \cite{moltmaker}. While \cite{moltmaker} focused on biframed spherical knotoids, in this paper we will need to work with biframed \textit{planar} knotoids. In this subsection we define these, and give their geometric realization.

\begin{definition}
\cite{moltmaker} A \textbf{framed knotoid} is an equivalence class of knotoid diagrams under the equivalence generated by the same moves that are allowed for knotoids, but with the first Reidemeister move $R1$ replaced by the \textbf{weakened} first Reidemeister move $R1'$, depicted in Figure \ref{fig:weakreid}. The \textbf{framing} of a framed knotoid is defined to be the writhe of any of its representative diagrams and is denoted $\text{fr}(K)$.
\begin{figure}[ht]
    \centering
    \includegraphics[width=.22\linewidth]{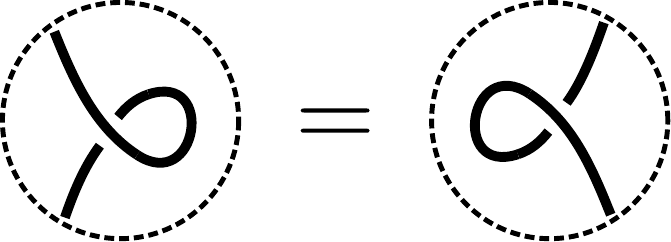}
    \caption{The weakened first Reidemeister move $R1'$.}
    \label{fig:weakreid}
\end{figure}
\end{definition}

\begin{definition}
Fix a vector $\vec{v}$ on the plane. A \textbf{biframed planar knotoid diagram} is a framed planar knotoid diagram $\phi:I\to S^2$ such that the tangent vectors of $\phi$ at $0,1\in I$ are parallel to $\vec{v}$. A \textbf{biframed planar knotoid} is an equivalence class of biframed planar knotoid diagrams under the equivalence generated by $R1',R2,R3$, and ambient isotopies of $\mathbb{R}^2$ that do not change the tangent vector directions of $\phi$ at $\{0,1\}\in I$, i.e.~smooth deformations of the biframed knotoid diagram such that the resulting diagram at each stage of the deformation is still biframed.
\end{definition}

Note that biframed knotoids are in particular framed. To justify the terminology `biframed' we show that restricting the tangent vectors of a knotoid at its leg and head amounts to attaching another integer to the knotoid. Later, we shall see from the geometric realization of biframed planar knotoids that this coframing really corresponds to a framing, namely of the auxiliary lines of a simple $H$-curve.

\begin{definition}\label{def:coframing}

Let $K$ be a biframed planar knotoid diagram. 
Then we define the \textbf{coframing} of $K$ by
\[
    \text{cofr}(K) = n_0-n_1.
\]
where $n_0$ is the winding number of $K$ with respect to the leg and $n_1$ is the winding number with respect to the head.

The \textbf{biframing} of $K$ is defined to be the pair $\left(\text{fr}(K),\text{cofr}(K)\right)$.
\end{definition}

In the above definition the knotoid diagram $K$ is oriented from leg to head and it should be clear that the winding numbers $n_0,n_1$ are integers so that the biframing is a pair of integers. 

\begin{lemma}\label{lm:cf_welldef}
The coframing of a biframed planar knotoid is well-defined.
\end{lemma}
\begin{proof}
Let $\kappa$ be a diagram of $K$ whose leg and head lie on the same directed line that has the direction of $\vec{v}$, with the leg coming before the head on this line. As any diagram for $K$ is related to some such $\kappa$ by an ambient isotopy preserving lines perpendicular to $\vec{v}$, it suffices to show that $\text{cofr}(K)$ is independent of the chosen diagram $\kappa$. 


Two choices of $\kappa$ can clearly be related by applications of $R1',R2,R3$, ambient isotopies fixing neighbourhoods of the leg and head, and applications of the `orbiting' move depicted in Figure \ref{fig:orbiting_move} (or its inverse), which is obtained by a full turn of the head of a biframed knotoid around its leg. Thus it suffices to show invariance of $\text{cofr}(K)$ under orbiting moves, as invariance under the others is immediate.
\begin{figure}[ht]
    \centering
    \includegraphics[width=.5\linewidth]{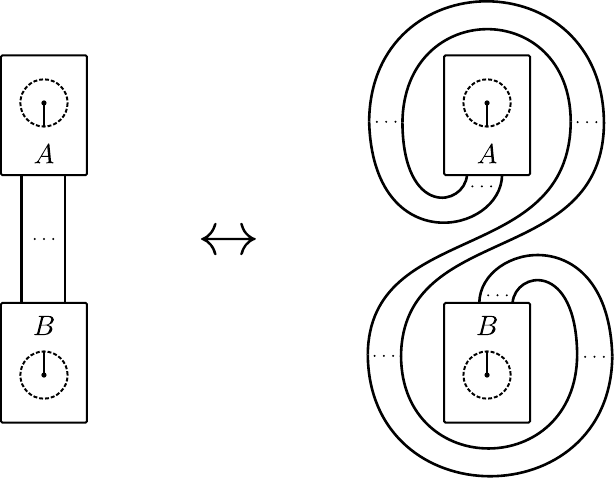}
    \caption{The `orbiting' move on a biframed knotoid diagram. Here the boxes labelled $A$ and $B$ are pieces of a knotoid diagram that are identical on both sides of the move, and two parallel strands with dots between them signify any number of parallel strands.}
    \label{fig:orbiting_move}
\end{figure}

As is shown in Figure \ref{fig:orbiting_reduction}, any orbiting move can be reduced to a `coframing exchange' between the end-points of a biframed knotoid diagram using ambient isotopies fixing neighbourhoods of the leg and head. Invariance of $\text{cofr}(K)$ under coframing exchange on the end-points is immediate, since the loops around the end-points on the right-hand side of Figure \ref{fig:orbiting_reduction} contribute $+1$ and $-1$ to $\text{cofr}(K)$ respectively.
\begin{figure}[ht]
    \centering
    \includegraphics[width=.9\linewidth]{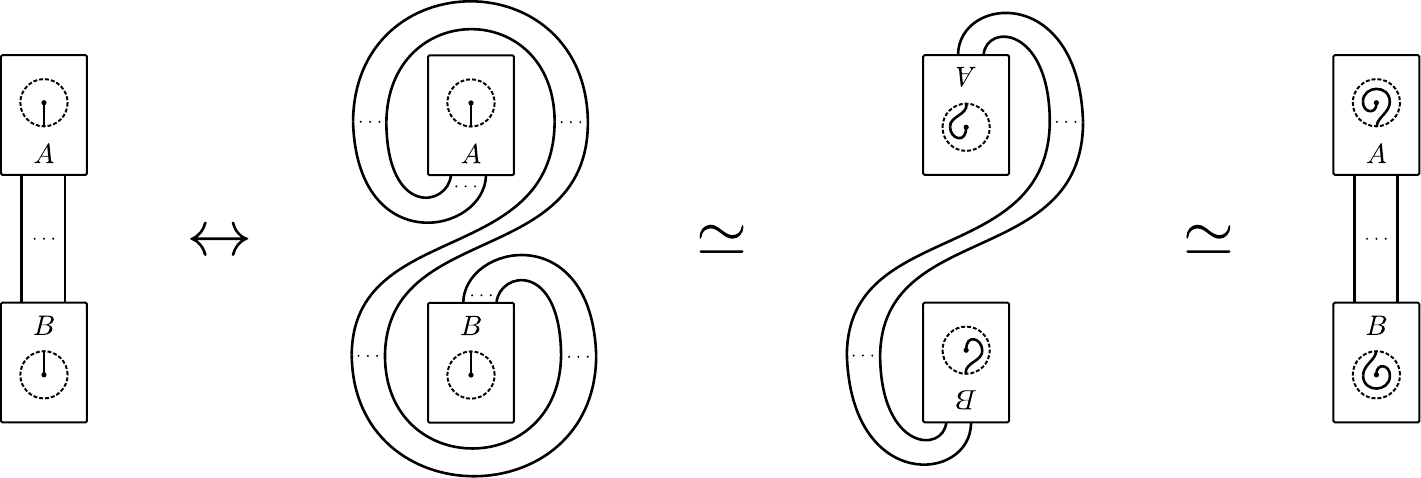}
    \caption{Sequence of biframed knotoid diagram equivalences showing that an arbitrary orbiting move is equivalent to an coframing exchange (right) on the end-points.}
    \label{fig:orbiting_reduction}
\end{figure}
\end{proof}

\begin{remark}
From Definition \ref{def:coframing} and the proof of Lemma \ref{lm:cf_welldef} it is clear that biframed knotoids can equivalently be described as knotoid diagrams whose leg and head must lie at fixed points $p,q\in\mathbb{R}^2$ and whose tangent vectors at $p,q$ lie in the direction $q-p$, up to the equivalence generated by $R1',R2,R3$, ambient isotopy relative to small neighbourhoods of $p,q$, and the \textbf{coframing identities}, depicted in Figure \ref{fig:coframing_identities}. This is the perspective on biframed knotoids taken in \cite{moltmaker}.


Under this approach orbiting moves are no longer allowed, as these are ambient isotopies that move the head or leg of a diagram. So to ensure that Lemma \ref{lm:cf_welldef} still holds the coframing identities must be imposed separately.
This approach has the advantage of reducing the complexity introduced by the coframing to two simple diagrammatic moves. This can be helpful for proving that some quantity associated to a knotoid diagram is an invariant of biframed knotoids.
\begin{figure}[ht]
    \centering
    \includegraphics[width=.7\linewidth]{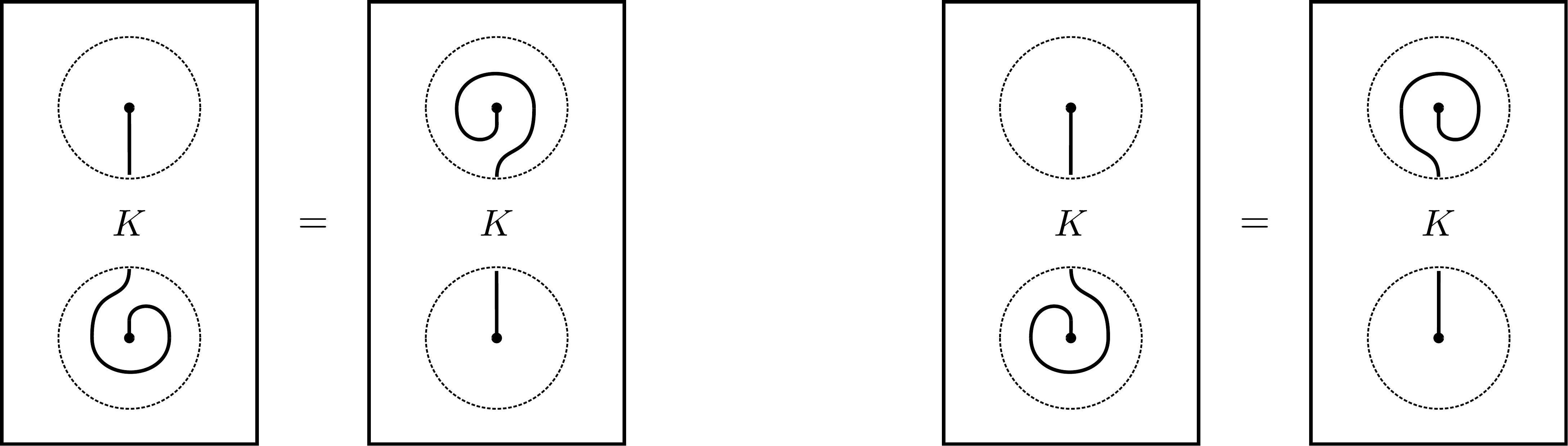}
    \caption{\cite{moltmaker} The coframing identities. Here `$K$' denotes the rest of some knotoid diagram that is understood to be identical on both sides of the identities.}
    \label{fig:coframing_identities}
\end{figure}
\end{remark}

\begin{remark}\label{rk:coframing_comp}
The coframing of a biframed knotoid diagram is clearly invariant under crossing changes. Thus generally the easiest way to compute coframing by hand is to apply crossing changes and knotoid diagram equivalences to a diagram until it is a trivial biframed knotoid, for which the coframing can easily be read off from the diagram.
\end{remark}

As the terminology suggests, the information of a biframed knotoid is equivalent to that of a knotoid with two integers attached:

\begin{lemma}\label{lm:diagrambijection}
There is a bijection
\begin{align}
    \{\text{Biframed planar knotoids}\} &\leftrightarrow \{\text{Planar knotoids}\}\times \mathbb{Z}^2 \\
    K &\mapsto \left(K,\text{fr}(K),\text{cofr}(K)\right). \nonumber
\end{align}
\end{lemma}

\begin{proof}
This proof is identical to that of the analogous statement for biframed spherical knotoids, which was given in \cite{moltmaker}.
\end{proof}

In subsequent sections we will need to work with explicitly oriented biframed planar knotoids. One caveat with doing so is that, given an oriented biframed knotoid diagram $K$, its `reverse' $-K$ is not well-defined. Namely, simply reverting the orientation of such a diagram to obtain $-K$ changes the tangent directions at its endpoints, so that $-K$ is no longer a biframed knotoid diagram. Instead, we define $-K$ as follows:

\begin{definition}\label{def:reversal}
Let $K$ be an oriented biframed planar knotoid diagram. We define its \textbf{reverse} $-K$ to be the biframed planar knotoid diagram given by reversing the orientation of $K$ and adding `hooks' to the end-points as depicted in Figure \ref{fig:hooks} to ensure the tangent vectors of $-K$ at the end-points are in the right direction. The hooks in Figure \ref{fig:hooks} are added in such a way that $-K$ has the same coframing as $K$. In fact the hooks can be added another way, instead attaching the other end of each hook to the end-points of $\overline{K}$. This gives another description of $-K$ with coframing equal to that of $K$. This alternative description is seen to be equivalent to the definition of $-K$ given in Figure \ref{fig:hooks} by applying the orbiting isotopy from Figure \ref{fig:orbiting_move}.
\begin{figure}[h]
    \centering
    \includegraphics[width=.5\linewidth]{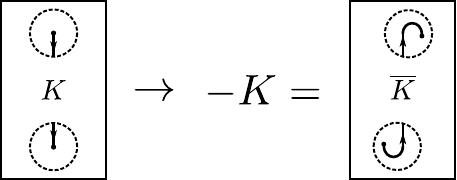}
    \caption{Definition of $-K$. Here the box labeled `$K$' denotes the portion of $K$ that is not pictured, and similarly $\overline{K}$ denotes the diagram $K$ with reversed orientation.}
    \label{fig:hooks}
\end{figure}
\end{definition}

To conclude this section we discuss the geometric realization of biframed planar knotoids. In doing so we will see how coframing, now expressed in terms of winding numbers, indeed corresponds to a framing in the same sense as for framed knots.

\begin{definition}
Let $\overline{H}$ be the topological space formed by thickening the edges of the graph $H$ depicted in Figure \ref{fig:H_graph} into three ribbons, as depicted in Figure \ref{fig:Hbar}. Also fix a direction $\vec{v}\in \mathbb{R}^2$.
\begin{figure}[ht]
    \centering
    \includegraphics[width=0.8\linewidth]{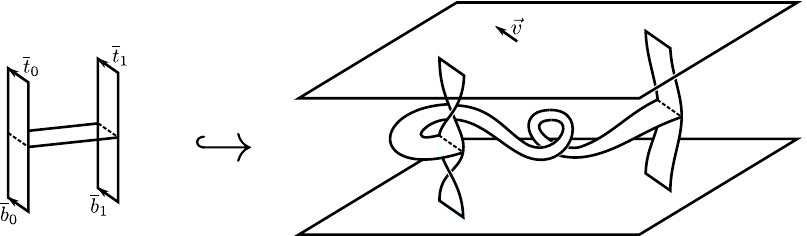}
    \caption{An embedding of $\overline{H}$ (left) yielding a biframed $H$-curve (right).}
    \label{fig:Hbar}
\end{figure}

A \textbf{biframed $H$-curve} is an embedding of $\overline{H}$ into $\mathbb{R}^2\times[-1,1]$ such that the images of the directed line pieces $\overline{t}_0$ and $\overline{t}_1$ depicted in Figure \ref{fig:Hbar} are line pieces in $\mathbb{R}^2\times\{1\}$ in the direction of $\vec{v}$, and the directed line pieces $\overline{b}_0$ and $\overline{b}_1$ are line pieces in $\mathbb{R}^2\times\{-1\}$ in the direction of $\vec{v}$.

We say two $H$-curves $A,B$ are \textbf{equivalent} if they can be related by a label-preserving ambient isotopy $F:(\mathbb{R}^2\times[-1,1])\times [0,1]\to \mathbb{R}^2\times[-1,1]$ such that the embedding $F(-,t)\circ A:\overline{H}\hookrightarrow \mathbb{R}^2\times[-1,1]$ defines a biframed $H$-curve for all $t\in[0,1]$.

In analogy with Definition \ref{def:H-curve}, we refer to the images of the vertical ribbons in the depiction of $\overline{H}$ given in Figure \ref{fig:Hbar} as the \textbf{auxiliary ribbons} of a biframed $H$-curve. 
\end{definition}

As with biframed planar knotoids, we assign a framing and coframing to biframed $H$-curves. It suffices to do so for the subclass of `standard' biframed $H$-curves.

\begin{definition}\label{def:standardHcurve}
A biframed $H$-curve $A$ is said to be \textbf{standard} if the following conditions hold:
\begin{itemize}
    \item The framed $(1,1)$-tangles corresponding to the auxiliary ribbons of $A$ are trivial and unframed (i.e.~have framing $0$).
    \item The auxiliary ribbons of $A$ are in manifestly unframed form, i.e.~are both of the form $[-\epsilon,\epsilon]\times [-1,1]$ where $[-\epsilon,\epsilon]\times \{i\}$ lies in $\mathbb{R}^2\times \{i\}$ along the direction of $\vec{v}$ for $i\in\{-1,1\}$.
    \item Let $\overline{r}$ denote the horizontal ribbon in the depiction of $\overline{H}$ in Figure \ref{fig:Hbar}, and let $a_0$ and $a_1$ be its attaching arcs to the vertical ribbons on the left and right, respectively. Then $A(a_i)\subseteq \mathbb{R}^2\times \{0\}$ for $i\in\{0,1\}$, and some neighbourhoods in $A(\overline{r})$ of $A(a_0)$ and $A(a_1)$ lie in $\mathbb{R}^2\times \{0\}$.
    \item Let $r$ denote one of the horizontal boundary components of $r$. Then the tangents to $A(r)$ at $A(a_i)$ are perpendicular to $A(a_i)$ for $i\in\{0,1\}$.
\end{itemize}
\end{definition}

\begin{definition}
Let $A$ be a standard biframed $H$-curve. Parametrize the arc $A(r)$ by $A(r)_t$ for $t\in[0,1]$ such that $A(r)_i\in A(a_i)$ for $i\in\{0,1\}$. For $t\in[0,1]$ we define $v_t$ to be the vector at $A(r)_t$ perpendicular to $A(r)$ pointing into $A(\overline{r})$. Running along $A(r)$ from $A(a_0)$ to $A(a_1)$, the collection $\{v_t\}_{t\in[0,1]}$ defines a path in $SO(2)$, which is a loop by the third assumption in Definition \ref{def:standardHcurve}. The corresponding element of $\pi_1(SO(2))\cong \mathbb{Z}$ is defined to be the \textbf{framing} of $A$, and is denoted $\text{fr}(A)$.

Next consider for all $t\in[0,1]$ the vectors in $\mathbb{R}^2\times [-1,1]$ from $A(r)_0$ and $A(r)_1$ to $A(r)_t$. Project these vectors onto $\mathbb{R}^2\times\{0\}$ and normalize them. This defines paths $p_0$ and $p_1$ in $SO(2)$ as before. If the tangents to $A(r)$ from the fourth assumption in Definition \ref{def:standardHcurve} are in the same direction in $\mathbb{R}^2\times\{0\}$, then these paths are loops. In this case, let $n_0$ and $n_1$ be the respective corresponding elements of $\pi_1(SO(2))\cong \mathbb{Z}$. Then we define the \textbf{coframing} of $A$, denoted $\text{cofr}(A)$, to be
\[
    \text{cofr}(A) \defeq n_0-n_1.
\]
If the tangents to $A(r)$ from the fourth assumption in Definition \ref{def:standardHcurve} have opposite directions, $\text{cofr}(A)$ is not defined. If $\text{cofr}(A)$ is defined, the \textbf{biframing} of $A$ is defined to be the ordered pair $(\text{fr}(A),\text{cofr}(A))$.
\end{definition}

\begin{remark}
Note that any biframed $H$-curve $A$ whose auxiliary ribbons are trivial and unframed $(1,1)$-tangles can easily be brought into standard form. Note also that the framing and coframing of a standard form of $A$ are independent of the chosen standard form. Indeed, this holds for the framing because the relative twisting between $A(a_0)$ and $A(a_1)$ when transforming an $H$-curve into standard form is clearly independent of the chosen standard form. The statement for coframing follows from reasoning analogous to the proof of Lemma \ref{lm:cf_welldef}. This observation allows us to define the framing of \textit{any} biframed $H$-curve $A$ whose auxiliary ribbons are trivial and unframed, namely as that of any standard biframed $H$-curve equivalent to $A$. We similarly define the coframing of $A$ to be that of a standard biframed $H$-curve equivalent to $A$, if the coframing of the latter is defined.
\end{remark}

Finally, we define the objects of interest to us in the context of biframed planar knotoids:

\begin{definition}
A biframed $H$-curve $A$ is said to be \textbf{simple} if both of the framed $(1,1)$-tangles defined by its auxiliary ribbons are trivial and unframed, and if the coframing of $A$ is defined.
\end{definition}

In particular the biframed $H$-curve depicted in Figure \ref{fig:Hbar} is simple, and has biframing $(-1,-1)$.

\begin{theorem}
Simple biframed $H$-curves form the geometric realization of biframed planar knotoids.
\end{theorem}

\begin{proof}
Let $A$ be a simple biframed $H$-curve. We can contract the ribbons of $A(\overline{H})$ to obtain a canonical simple $H$-curve $A^\circ$ from $A$. We consider the following map:
\begin{align*}
    \psi: \{\text{Biframed simple $H$-curves}\} &\to \{\text{Simple $H$-curves}\}\times \mathbb{Z}^2\\
    A &\mapsto \left(A^\circ,\text{fr}(A),\text{cofr}(A)\right).
\end{align*}
We claim $\psi$ is a bijection. The proof of this claim is identical to that for the analogous statement for simple biframed theta-curves, see \cite{moltmaker}. In short, surjectivity is clear, and for injectivity an equivalence $A^\circ\simeq B^\circ$ easily extends to an equivalence $A\simeq B$ up to framing twists and coframing loops around the auxiliary ribbons, which must cancel if $\text{fr}(A)=\text{fr}(B)$ and $\text{cofr}(A)=\text{cofr}(B)$. This is done e.g.~by orbiting moves on the auxiliary ribbons analogous to Figure \ref{fig:orbiting_move}, but seen as ambient isotopies of $\mathbb{R}^2\times[-1,1]$. Combining the bijection $\psi$ with the bijections from Proposition \ref{prop:planar_realization} and Lemma \ref{lm:diagrambijection} we obtain a bijection between biframed simple $H$-curves and biframed planar knotoids, as required.
\end{proof}

\section{Knot Measures}\label{sec:measures}

In this section we describe how knotoids and invariants of knotoids can be used to quantify how much a smooth open curve in $\mathbb{R}^3$ is knotted.\\

The primary application of knotoids to quantifying the knottedness of open curves lies in the observation that projecting an open curve onto a plane typically results in a knotoid diagram. More precisely:

\begin{definition}\label{def:projection}
Let $C$ be a smooth open curve in $\mathbb{R}^3$. Then for a direction specified by a vector $\vec{v}$ on the unit sphere $S^2\subseteq\mathbb{R}^3$ we define $C_{\vec{v}}$ to be the diagram given by projecting $C$ onto a plane with normal vector $\vec{v}$ and recording crossing information.
\end{definition}

The procedure of Definition \ref{def:projection} produces a knotoid diagram unless the projection specified by $\vec{v}$ happens to project an endpoint of the open curve exactly onto one of its strands, or creates cusps or triple points in the diagram. Clearly such problematic projection directions form a measure zero subset of $S^2$ with respect to the standard measure on the unit sphere $S^2\subseteq \mathbb{R}^3$, so that we are justified in saying that projecting an open curve `almost always' results in a knotoid diagram.

The observation motivating Definition \ref{def:projection} also work to produce knot diagrams from closed smooth curves, but by Reidemeister's theorem all knot diagrams $C_{\vec{v}}$ obtained in this way are equivalent. This is not the case for knotoids: an example of one open curve projecting to two different spherical knotoids is depicted in Figure \ref{fig:projection}. This issue stems from the fact that finding knottedness in open curves is an ill-posed problem, since all smooth open curves in $\mathbb{R}^3$ are ambient isotopic.

\begin{figure}[ht]
    \centering
    \includegraphics[width=.4\linewidth]{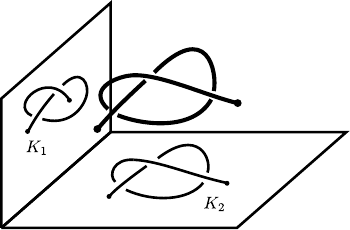}
    \caption{An open curve projected into different directions, yielding different knotoids $K_1$ and $K_2$. The knotoid $K_1$ is equivalent to the simplest nontrivial spherical knotoid $2_1$, while $K_2$ is equivalent to the trivial knotoid.}
    \label{fig:projection}
\end{figure}

knottedness of open curves can therefore not be quantified by an isotopy invariant, as one does for closed curves. Several methods to circumvent this problem have been proposed. The most basic is to reduce to the situation of closed curves by taking a `statistical closure' or finding the `dominant knot' in an open curve. This approach is discussed in \cite{sulkowska2012conservation,jamroz2015knotprot,gabrovsek2021invariant} for example.

\begin{definition}
Let $C$ be a smooth open curve in $\mathbb{R}^3$. Place $C$ inside a very large sphere $S^2$. For every point $\vec{v}\in S^2$ let $C_{\vec{v}}$ denote the knot obtained by adding straight lines from the end-points of $C$ to $\vec{v}$, if this indeed defines a knot. (Like Definition \ref{def:projection}, this defines a knot for almost all $\vec{v}\in S^2$.) We define the \textbf{dominant knot} of $C$ to be the knot that has the highest probability of being obtained in this fashion when a point $\vec{v}\in S^2$ is selected with uniform probability.
\end{definition}

The approach of finding a dominant knot is reasonable if our open curve resembles a long knot, but may fail to accurately model the topology of a generic open curve whose endpoints may lie in the middle of a highly tangled region of the curve. Such open curves do not have a canonical closure, and it may not be ideal to model them using a knot. A clear solution to this last objection is instead to work with a `dominant knotoid':

\begin{definition}\label{def:dom_knotoid}
For a smooth open curve $C$ in $\mathbb{R}^3$ and a vector $\vec{v}$ on the unit sphere, let $C_{\vec{v}}$ be as in Definition \ref{def:projection}. Then the \textbf{dominant knotoid} of $C$ is the knotoid that is most likely to be equivalent to $C_{\vec{v}}$ when $\vec{v}$ is selected randomly from $S^2$ with uniform probability.
\end{definition}

While the dominant knotoid may of an open curve $C$ may better model the topology of $C$ than its open knot, picking only one of the knotoids arising from an open curve is still somewhat ad hoc. A more uniform approach is taken in \cite{goundaroulis2017topological,dorier2018knoto} where $S^2$ is coloured according to the knotoid type of $C_{\vec{v}}$ for all $\vec{v}\in S^2$. While this method is certainly comprehensive in encoding the knottedness of $C$, it is extremely difficult to implement. Indeed, for certain points $\vec{v}\in S^2$ the knotoid diagram $C_{\vec{v}}$ can have a particularly large amount of crossings, making the knotoid type of $C_{\vec{v}}$ difficult to classify. More-over the classification of knotoids is currently limited: the classification of prime planar knotoids with up to 5 crossings is not yet complete. As a result this method, while very comprehensive, is too difficult to implement for most applications.\\

A rather different approach is that taken in \cite{panagiotou2020,panagiotou2021}, where the knottedness of an open curve $C$ is quantified by a function that depends continuously on the coordinates of $C$. This is the approach that we further investigate here since it not particularly ad hoc and will generally result in measures of knottedness that are computable (by virtue of being continuous and hence susceptible to approximation). 

In order to generalize the approach taken in \cite{panagiotou2020,panagiotou2021} we define `knot measures' of open curves. These are meant to be the best thing one can hope for after noting that ambient isotopy invariants cannot detect any information of open curves.

\begin{definition}\label{def:knot_measures}
A \textbf{knot measure} is a function $\varphi$ of open curves $C$ taking values in some real vector space, such that:
\begin{enumerate}
    \item $\varphi$ is continuous with respect to the topology induced by the Hausdorff metric on the space of images of open curves.
    \item There exists a knot invariant $\overline{\varphi}$ that $\varphi$ extends to in the following sense: As we bring the ends of an open curve $C$ together to form of knot $K$, the value of $\varphi(C)$ must converge to $\overline{\varphi}(K)$. As a short-hand:
    \[
        \lim_{C\to K} \varphi(C) = \overline{\varphi}(K).
    \]
\end{enumerate}
\end{definition}

In the context of knot measures, the applicability of spherical knotoids to quantifying knottedness of open curves is now expressed by the following theorem:

\begin{theorem}\label{thm:knot_measures}
Every spherical knotoid invariant $\varphi$ gives rise to a knot measure by defining
\[
        \varphi(C) \defeq \frac{1}{4\pi} \int_{\vec{v}\in S^2-X} \varphi(C_{\vec{v}})\, dS,
\]
on open curves $C$. Here $X$ is the measure zero subset of $S^2$ for which $C_{\vec{v}}$ is not a valid knotoid diagram.
\end{theorem}

\begin{proof}
We check both defining properties of a knot measure, in order:

\begin{enumerate}
    \item This part of the proof is adapted from \cite{panagiotou2020}, where the case in which $\varphi$ is the Kauffman bracket is discussed. First we restrict to the case where $C=E_n$ lies in the class of polygonal chains consisting of $n$ edges. In this case $(E_n)_{\vec{v}}$ is one of finitely many knotoids. Let us list these knotoids as $\{K_1,\dots,K_m\}$. Then
\[
    \varphi(E_n) = \frac{1}{4\pi} \int_{\vec{v}\in S^2-X} \varphi\left((E_n)_{\vec{v}}\right)\, dS
    = \sum_{i=1}^m p_i \varphi(K_i),
\]
where $p_i=\mathbb{P}\left((E_n)_{\vec{v}}=K_i\right)$ when $\vec{v}$ is randomly sampled from $S^2-X$ with a uniform probability distribution. The proof of \cite[Lemma.~3.1]{panagiotou2020} shows that each $p_i$ is a uniformly continuous function of the coordinates of $E_n$.

Let $V$ be the free vector space over the set $\{\varphi(K_i)\}_{i=1}^m$. Then clearly $V$ is finite-dimensional, namely its dimension is bounded by $m$. Consider the standard Euclidean norm $\|\cdot\|$ on $V$. With respect to this norm we have
\[
    \|\varphi(E_n)\| = \sqrt{ \sum_{i=1}^m p_i{}^2 }.
\]
Thus since each $p_i$ is a uniformly continuous function of the coordinates of $E_n$ we conclude that $\varphi(E_n)$ is too. So the result we wish to prove holds when restricted to the class of polygonal chains with $n$ edges, for any $n$.

Let $\mathcal{P}$ be the set of all polygonal open chains. If $C$ is any smooth open curve then we can make a sequence of polygonal approximations $\{C_n\}_{n\in\mathbb{N}}$ of $C$ such that for all $\varepsilon>0$ there exists an $N\in \mathbb{N}$ so that $C_n$ is contained in a tube of radius $\varepsilon$ around $C$ for all $n>N$. Thus $\mathcal{P}$ is dense in the space of all open curves. Let $W$ be the vector space over $\mathbb{R}$ spanned by $\{\varphi(K)\}_{K\in\mathcal{K}}$, where $\mathcal{K}$ is the set of all spherical knotoids. Let $B$ be a (Hamel) basis for $W$ contained in $\{\varphi(K)\}_{K\in\mathcal{K}}$. The space $W$ may be infinite-dimensional, but since $\mathcal{K}$ is countably infinite, $B$ is finite or countably infinite. Consider the $l^2$ norm $\|\cdot\|_2$ on $W$ with respect to $B$. If $W$ is infinite-dimensional then $W$ is not complete with respect to this norm. However, by construction we have $\|\varphi(C)\|_2\leq 1$ for all $C$, so that the image of $\varphi$ lies inside the copy of $l^2$ contained in $W$. Hence we can take the codomain of $\varphi$ to be complete. By uniform continuity of $\varphi$ on $\mathcal{P}$, density of $\mathcal{P}$, and completeness of the codomain of $\varphi$, we conclude that $\varphi$ is continuous at $C$ (see, for example, \cite[Ch.~24]{erdman2005problemtext}).
    
    \item This part of the proof follows from canonically associating a spherical knotoid $K^\bullet$ to any knot $K$ \cite{turaev2012knotoids}. We then define $\overline{\varphi}$ by
    \[
        \overline{\varphi}(K) \defeq \varphi(K^\bullet).
    \]
    The knotoid $K^\bullet$ is defined by taking a diagram of $K$ and removing a small segment from one of its arcs, away from the crossings of $K$. To see that this is well-defined we must show that $K^\bullet$ is independent of the chosen arc. The proof of this is depicted in Figure \ref{fig:K_bullet}: The removed arcs can always be moved over or under one crossing by dragging the obstructing strand of the crossing along the back of $S^2$ as shown. Note that the proof of this fact also goes through for framed knotoids, since the loops undone by by the first Reidemeister move in the last equality of Figure \ref{fig:K_bullet} are oppositely oriented.
    
    \begin{figure}[ht]
        \centering
        \includegraphics[width=.8\linewidth]{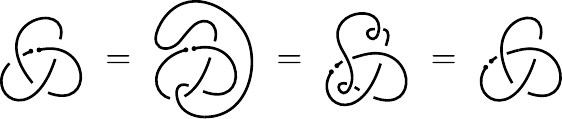}
        \caption{Showing well-definedness of $K^\bullet$.}
        \label{fig:K_bullet}
    \end{figure}
    
    Clearly $\overline{\varphi}$ is now a knot invariant, since equivalent knots $K$ and $L$ give rise to equivalent knotoids $K^\bullet$ and $L^\bullet$.
    
    Finally, we prove $\overline{\varphi}$ extends $\varphi$. Say that we bring the ends of $C$ together to form a closed curve $K$. By Reidemeister's theorem all diagrams $K_{\vec{v}}$ are equivalent to some knot diagram $\kappa$. Clearly as the ends of $C$ are brought together, $\mathbb{P}(C_{\vec{v}}=\kappa^\bullet)$ converges to 1. Hence
    \[
        \lim_{C\to K} \varphi(C) = \frac{1}{4\pi} \int_{\vec{v}\in S^2-X} \varphi\left(\kappa^\bullet \right)\, dS = \frac{1}{4\pi} 4\pi \varphi(\kappa^\bullet) = \varphi(\kappa^\bullet) = \overline{\varphi}(K),
    \]
    as required.
\end{enumerate}
\end{proof}

\begin{remark}
The requirement that $\varphi$ is a \textit{spherical} knotoid invariant is necessary. Namely $K^\bullet$ is not well-defined as a planar knotoid and hence planar invariants $\varphi$ do not extend to a well-defined knot invariant $\overline{\varphi}$. The fact that $K^\bullet$ is not well-defined as a planar knotoid is exemplified by the two knotoids at the far left- and right-hand sides of Figure \ref{fig:K_bullet}: these are known to be inequivalent as planar knotoids, tabulated in \cite{goundaroulis2019systematic} as $3_{16}$ and $3_1$ respectively.

However, for a planar knotoid invariant $\varphi$ the first part of the proof of theorem \ref{thm:knot_measures} still goes through. So in this case we still obtain a continuous measure of knottedness; just one that doesn't extend to an isotopy invariant on closed curves. 

Noting that there is a surjection from planar knotoids to spherical knotoids (induced by the one-point compactification of $\mathbb{R}^2$), we conclude knot measures associated to planar knotoid invariants can generally contain more information about the shape of a knotted open curve than those coming from spherical invariants. However, we see that this information does not have the character of an ambient isotopy invariant, and is therefore of a more geometric nature.
\end{remark}

Note that the approach taken in \cite{goundaroulis2017topological,dorier2018knoto} where $S^2$ is colored according to the spherical knotoid type of $C_{\vec{v}}$ determines the knot measures constructed in Theorem \ref{thm:knot_measures}. Indeed, after integration over $S^2$ this approach yields the case of Theorem \ref{thm:knot_measures} when $\varphi$ is the trivially complete knotoid invariant given by 
\[
    \varphi(K)=\text{[Knotoid type of $K$]}.
\]

\section{Universal Quantum Invariants of Planar Knotoids}\label{sec:invariants}

In this section we define the universal quantum invariant of biframed knotoids associated to any ribbon Hopf algebra $A$. We show how all of the quantum invariants of planar knotoids discussed in \cite{moltmaker} can be recovered from these universal invariants. Afterwards we follow the approach to universal invariants seen in \cite{bar2021} by working with the algebra $\mathbb{D}$, which is a ribbon Hopf algebra over $\mathbb{Q}[\epsilon]\llbracket \hbar\rrbracket$ related to quantum $U(\mathfrak{sl}_2)$. In \cite{bar2021} a Mathematica \cite{Wolf} implementation is given for efficiently computing the universal invariant associated to $\mathbb{D}$ up to fixed order in $\epsilon$. We end this section by applying this Mathematica implementation to biframed knotoid diagrams, in order to carry out example computations for universal quantum invariants.

\subsection{Universal Quantum Invariants}\label{subsec:universal_invariants}


In this subsection we follow the notation from \cite{ohtsuki2002quantum}, as all the results for knots that we will generalize to knotoids in this subsection can be found there. Throughout, we assume without loss of generality that the tangent vectors at the endpoints of all biframed planar knotoids are directed vertically downwards, and we let $A$ be a ribbon Hopf algebra. That is, $A$ has morphisms $(m,\Delta,i,\epsilon,S)$ endowing it with the structure of a Hopf algebra, has a quasitriangular structure $\mathcal{R}\in A\otimes A$, and has a central `ribbon element' $v$ such that
\begin{align*}
    & v^2 = S(u)\cdot u,\\
    & \Delta(v) = (v\otimes v)\cdot (\mathcal{R}_{21}\mathcal{R})^{-1},\\
    & S(v)=v,\\
    & \epsilon(v)=1.
\end{align*}
where, writing $\mathcal{R}=\sum_i \alpha_i\otimes \beta_i$, we denote $\mathcal{R}_{21}=\sum_i \beta_i\otimes \alpha_i$ and define $u\in A$ to be given by $u=\sum_i S(\beta_i)\cdot \alpha_i$. Similarly, we will write $\mathcal{R}^{-1}\in A\otimes A$ as $\mathcal{R}^{-1} = \sum_i \alpha_i'\otimes \beta_i'$.

Given a ribbon Hopf algebra $A$ we will construct an invariant $Q^{A;\star}$ of oriented biframed planar knotoids, taking values in $A$, called the `universal quantum invariant associated to $A$'. We will describe the construction of this invariant in terms of elementary knotoid diagram pieces.

\begin{definition}
The \textbf{elementary knotoid diagram pieces} are the portions of knotoid diagram depicted in Figure \ref{fig:elementary_knotoid_pieces}, with either orientation allowed for the unoriented pieces. 

\begin{figure}[ht]
    \centering
    \includegraphics[width=\linewidth]{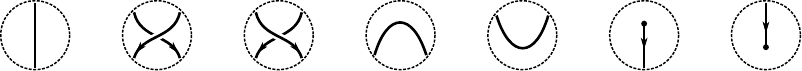}
    \caption{The elementary knotoid diagram pieces.}
    \label{fig:elementary_knotoid_pieces}
\end{figure}

Let $D_1$ and $D_2$ be portions of knotoid diagram all of whose open ends (excluding any knotoid end-points) are either at the top or bottom of the diagram, i.e.~let $D_1$ and $D_2$ be tangles that may be decorated with the head and/or leg of a knotoid. We define their \textbf{tensor product} $D_1\otimes D_2$ to be their horizontal juxtaposition. If $D_1$ has $n$ open ends at the bottom and $D_2$ has $n$ open ends at its top, then we define their \textbf{composition} $D_1\circ D_2$ to be the diagram portion obtained by placing $D_1$ above $D_2$ and gluing the open ends at the bottom of $D_1$ to those at the top of $D_2$, in order\footnote{This convention for the notation of composition essentially means that we read a composition of diagram portions `from bottom to top'.}. Note that in this way and biframed knotoid diagram can be decomposed into a finite sequence of compositions and tensor products of elementary knotoid diagram pieces (again, assuming without loss of generality that the tangent vectors at its endpoints are directed vertically downwards).
\end{definition}

\begin{definition}
Let $K$ be an oriented biframed knotoid diagram, and $A$ a ribbon Hopf algebra. We define the \textbf{universal quantum invariant $Q^{A;\star}$ associated to $A$} by decomposing $K$ into compositions and tensor products of elementary knotoid diagram pieces. We evaluate $Q^{A;\star}$ on these pieces via Figure \ref{fig:universal_invariant}. We then glue the values of $Q^{A;\star}$ on these elementary pieces back together according to $K$. To find $Q^{A;\star}(K)$ we then run through the resulting diagram from leg to head, multiplying the elements of $A$ placed on $K$ together as we encounter them, and finally sum over all the indices of copies of $\mathcal{R}$ and $\mathcal{R}^{-1}$. The result is an element of $A$ which we define to be $Q^{A;\star}(K)$.

\begin{figure}[ht]
    \centering
    \includegraphics[width=.9\linewidth]{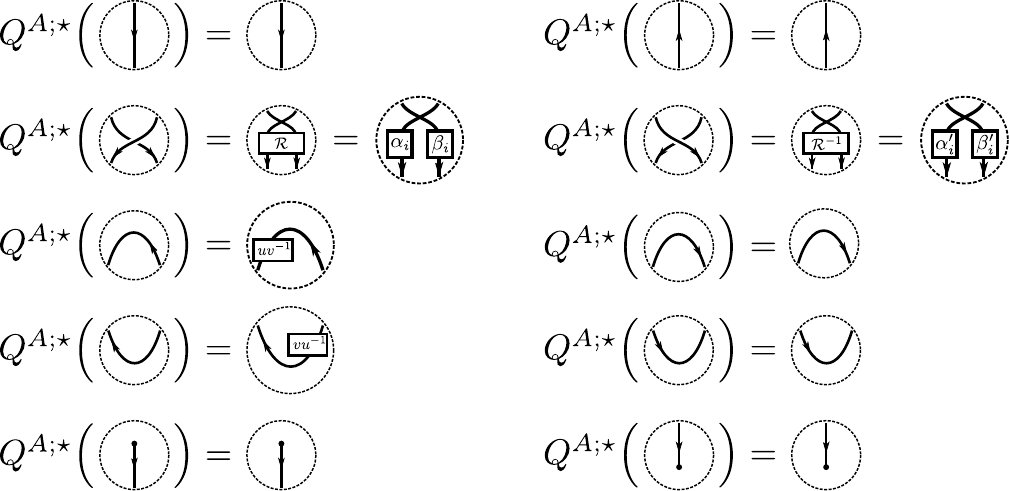}
    \caption{The value of $Q^{A;\star}$ on elementary knotoid diagram pieces.}
    \label{fig:universal_invariant}
\end{figure}

\end{definition}

\begin{example}\label{ex:universal_quantum_example}

The definition of $Q^{A;\star}$ is illustrated for an example knotoid in Figure \ref{fig:universal_quantum_example}.

\begin{figure}[ht]
    \centering
    \includegraphics[width=.85\linewidth]{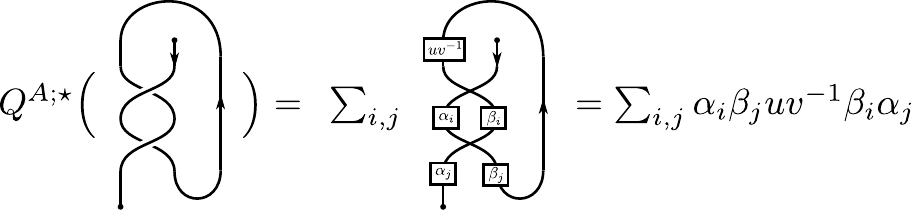}
    \caption{Computing $Q^{A;\star}(K)$ for an example knotoid $K$.}
    \label{fig:universal_quantum_example}
\end{figure}

\end{example}

\begin{remark}\label{rk:codomain}
In the case of knots, the universal quantum invariant takes values in $A/I$ where $I$ is the vector subspace spanned by elements of the form $xy-yx$ for $x,y\in A$. Namely to compute the universal quantum invariant of knots, one chooses a place to start running through the knot diagram. To ensure the resulting element of $A$ is independent of the chosen starting point, one must therefore quotient by $I$. This is not the case for knotoids since they, like tangles, have a canonical starting point, namely the leg. If we do consider $Q^{A;\star}(K)$ as an element of $A/I$ for a knotoid $K$, then the result is easily seen to be equal to the universal quantum invariant of the virtual closure of $K$, seen as a rotational virtual knot\footnote{Here the virtual closure is defined by taking a diagram for $K$ such that the direction from its head to its leg is equal to its tangent direction at the end-points, and adding the straight line from head to leg to $K$ interpreting every crossing of this line with $K$ as virtual.}. See \cite{kauffman2015rotational} for details on universal quantum invariants of rotational virtual knots. 
\end{remark}

\begin{lemma}\label{lm:universal_invariance}
The universal quantum invariant $Q^{A;\star}$ is an invariant of oriented biframed planar knotoids. 
\end{lemma}

\begin{proof}

By the proof of the analogous statement for framed knots, it suffices to show invariance under the orbiting moves (recall Figure \ref{fig:orbiting_move}). This is a trivial computation: just evaluate $Q^{A;\star}$ on all bends and note that every strand obtains exactly two canceling factors $uv^{-1}$ and $vu^{-1}$.
\end{proof}

The quantum invariant $Q^{A;\star}$ is `universal' in the sense that it dominates the quantum invariants defined from a Hopf algebra representation using a Reshetikhin-Turaev construction \cite{moltmaker}, in analogy with quantum invariants of knots. To show this we first recall the definition of Reshetikhin-Turaev invariants of planar knotoids, phrased in the notation of \cite{ohtsuki2002quantum} for consistency.

\begin{definition}
Let $K$ be an oriented biframed knotoid diagram, $A$ a ribbon Hopf algebra, and $(V,\rho)$ a finite-dimensional representation of $A$ over a field $k$. Let $R\in\text{End}(V\otimes V)$ and $h\in\text{End}(V)$ be given by
\[
    R = \tau\circ (\rho\otimes \rho)(\mathcal{R})
    \qquad\text{ and }\qquad
    h = \rho(uv^{-1}),
\]
where $\tau:x\otimes y\mapsto y\otimes x$. Note that therefore $R(x\otimes y)=\sum_i \rho(\beta_i)(y)\otimes \rho(\alpha_i)(x)$. We further define the morphisms $n:V\otimes V^*\to k$, $n':V^*\otimes V\to k$, $u:k\to V^*\otimes V$, and $u':k\to V\otimes V^*$ as follows:
\begin{align*}
    & n(x\otimes f) = f(h(x)),
    \qquad\qquad
    n'(f\otimes x) = f(x), \\
    & u(1) = \sum_i e^i\otimes h^{-1}(e_i),
    \qquad\qquad
    u'(1) = \sum_i e_i\otimes e^i,
\end{align*}
where $\{e_i\}$ is any basis of $V$ and $\{e^i\}$ is the associated dual basis of $V^*$. Finally, we let $\eta:k\to V$ and $\epsilon:V\to k$ denote the generic linear maps given by $\eta(1)=\sum_i \eta^i e_i$ and $\epsilon(\sum_i \lambda^i e_i)=\sum_i \epsilon_i\lambda^i$ for arbitrary $\{\eta^i,\epsilon_i\}\subseteq k$.

Decompose $K$ into a sequence of compositions and tensor products of elementary knotoid diagram pieces. The \textbf{Reshetikhin-Turaev invariant $Q^{A;V}$} associated to $(A,V)$ is an element of $\text{Hom}(k,k)\cong k$. It is defined by associating the morphisms defined above to elementary knotoid diagram pieces via Figure \ref{fig:RT_invariant}, and composing the morphisms associated to portions $D_1,D_2$ of knotoid diagram according to the rules $Q^{A;V}(D_1\otimes D_2)=Q^{A;V}(D_1)\otimes Q^{A;V}(D_2)$ and $Q^{A;V}(D_1\circ D_2)=Q^{A;V}(D_1)\circ Q^{A;V}(D_2)$. Note that the result is indeed always an element of $\text{Hom}(k,k)$, since a biframed planar knotoid diagram has no open ends except for its leg and head, which represent morphisms to and from $k$ respectively. See \cite{moltmaker,ohtsuki2002quantum} for more details.

\begin{figure}[ht]
    \centering
    \includegraphics[width=.75\linewidth]{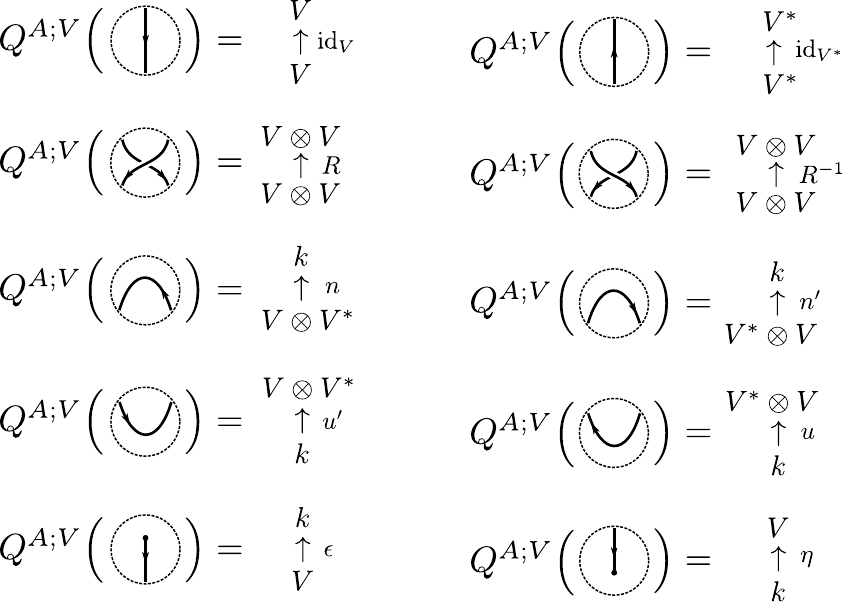}
    \caption{The value of $Q^{A;V}$ on elementary knotoid diagram pieces.}
    \label{fig:RT_invariant}
\end{figure}

\end{definition}

\begin{lemma}
The Reshetikhin-Turaev invariant $Q^{A;V}$ is an invariant of oriented biframed planar knotoids.
\end{lemma}

\begin{proof}
This can be proven analogously to the proof of Lemma \ref{lm:universal_invariance}. Alternatively, it immediately follows as a corollary from Proposition \ref{prop:recovery} below.
\end{proof}

\begin{example}\label{ex:RT_example}
The definition of $Q^{A;V}$ is illustrated in Figure \ref{fig:RT_example} for the same example knotoid as Example \ref{ex:universal_quantum_example}. From Figure \ref{fig:RT_example} we read off
\[
    Q^{A;V}(K) = n\circ (\text{id}\circ \epsilon\otimes \text{id})\circ (R\otimes \text{id})\circ (R\otimes \text{id})\circ (\eta\otimes u').
\]

\begin{figure}[ht]
    \centering
    \includegraphics[width=.48\linewidth]{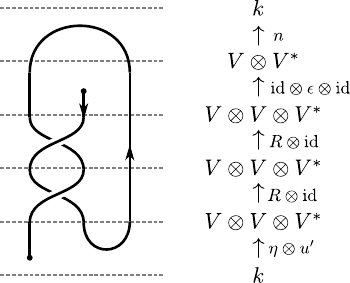}
    \caption{Computing $Q^{A;V}(K)$ for an example knotoid $K$.}
    \label{fig:RT_example}
\end{figure}

\end{example}

As the notation suggests, we should be able to recover $Q^{A;V}$ from $Q^{A;\star}$ by somehow substituting a representation $(V,\rho)$ of $A$ into $Q^{A;\star}$. This is made precise by the following proposition:

\begin{proposition}\label{prop:recovery}
Let $A$ be a ribbon Hopf algebra and $(V,\rho)$ a finite-dimensional representation of $A$. Then for $K$ a biframed knotoid we have
\begin{equation}\label{eq:recovery}
    Q^{A;V}(K) = \epsilon \circ \rho\left( Q^{A;\star}(K) \right) \circ \eta,
\end{equation}
where $\eta:k\to V$ and $\epsilon:V\to k$ are the linear maps used in the definition of $Q^{A;V}(K)$.
\end{proposition}

\begin{proof}
This proof is analogous to that for the analogous statement for knots, given in \cite{ohtsuki2002quantum}. Let $n$ be the dimension of $V$. We introduce a state sum formula, and note that it is equal to both sides of Equation \eqref{eq:recovery}. 

Pick a basis $\{e_i\}$ of $V$. To construct the state sum we will represent a linear map $f\in\text{Hom}(V^{\otimes n},V^{\otimes m})$ using tensor notation as $f_{a_1\dots a_n}^{b_1\dots b_m}$ so that $f(e_{a_1}\otimes\dots\otimes e_{a_n})=\sum_{b_1\dots b_m} f_{a_1\dots a_n}^{b_1\dots b_m} e_{b_1}\otimes\dots\otimes e_{b_m}$. We represent maps in $\text{Hom}\left(V^{\otimes n}\otimes (V^*)^{\otimes p},V^{\otimes m}\otimes (V^*)^{\otimes q}\right)$ similarly using the basis $\{e^i\}$ dual to $\{e_i\}$.

Picking a Morse decomposition for $K$, we see it as a decorated planar graph with vertices at its crossings and critical points of the vertical coordinate. We associate a label to each of its edges. A state of the diagram $K$ is an association of an element of $\{1,2,\dots,n\}$ to every label. We associate a weight $w(E)$ to every elementary knotoid diagram piece via Figure \ref{fig:state_pieces} where, in order to agree with the construction of $Q^{A;V}$ we define
\[
    n_{ij}\defeq h^j_i,\qquad
    n'_{ij}\defeq \delta_{ij},\qquad
    u^{ij}\defeq (h^{-1})^j_i,\qquad
    u'^{ij} = \delta_{ij}.
\]

\begin{figure}[ht]
    \centering
    \includegraphics[width=.8\linewidth]{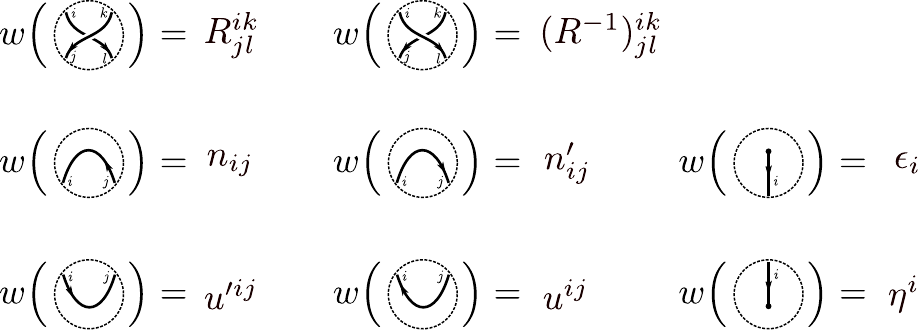}
    \caption{Weights of the elementary knotoid diagram pieces.}
    \label{fig:state_pieces}
\end{figure}

We define the state sum of $K$ to be
\[
    \sum_S \prod_E w(E),
\]
where the sum is over all states of $K$ and the product is over all elementary diagram pieces in a decomposition of $K$ into elementary pieces. Now, we have
\[
    Q^{A;V}(K) = \sum_S \prod_E w(E) = \epsilon \circ \rho\left( Q^{A;\star}(K) \right) \circ \eta.
\]
Here the first equality follows from expanding the constituents of $Q^{A;V}(K)$ into a sum over indices, and the second equality follows from noting that
\[
    R^{ij}_{kl} = \sum_m \rho(\alpha_m)^j_k\rho(\beta_m)^i_l
    \qquad\text{ and }\qquad
    h^i_j = \rho(uv^{-1})^i_j.
\]
See \cite[Ch.~4]{ohtsuki2002quantum} for further details.
\end{proof}

\begin{example}
To illustrate the state sum formula from the proof of Proposition \ref{prop:recovery} we compute it for the same example knotoid as in Examples \ref{ex:universal_quantum_example} and \ref{ex:RT_example}, and use the result to illustrate the proof of Proposition \ref{prop:recovery}. Label $K$ according to Figure \ref{fig:state_example}.

\begin{figure}[ht]
    \centering
    \includegraphics[width=.17\linewidth]{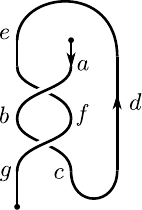}
    \caption{A generic state of an example knotoid $K$.}
    \label{fig:state_example}
\end{figure}

Using Figure \ref{fig:state_example} we compute
\[
    \sum_S \prod_E w(E) = \sum R^{ea}_{bf} R^{bf}_{gc} u'^{cd} n_{ed} \epsilon_a \eta^g.
\]
Filling in the definitions of $R$, $n$, and $u'$, we compute
\begin{align*}
    \sum_S \prod_E w(E) &= \sum \sum_{k,l} \rho(\alpha_k)^a_b \rho(\beta_k)^e_f \rho(\alpha_l)^f_g \rho(\beta_l)^b_c \delta_{cd} \rho(uv^{-1})^d_e \epsilon_a \eta^g\\
    &= \sum \sum_{k,l} \epsilon_a \rho(\alpha_k)^a_b \rho(\beta_l)^b_c \rho(uv^{-1})^c_e \rho(\beta_k)^e_f \rho(\alpha_l)^f_g \eta^g\\
    &= \sum \epsilon_a\,\, \rho\left( \sum_{k,l} \alpha_k \beta_l uv^{-1} \beta_k \alpha_l \right)^a_g\,\, \eta^g\\
    &= \sum \epsilon_a \rho(Q^{A;\star}(K))^a_g \eta^g\\
    &= \epsilon \circ \rho\left( Q^{A;\star}(K) \right) \circ \eta,
\end{align*}
where the fourth equality follows from Example \ref{ex:universal_quantum_example}. On the other hand, expanding $Q^{A;V}(K)$ from Example \ref{ex:RT_example} into index notation we compute
\begin{align*}
    Q^{A;V}(K) &= n\circ (\text{id}\circ \epsilon\otimes \text{id})\circ (R\otimes \text{id})\circ (R\otimes \text{id})\circ (\eta\otimes u')\\
    &= \sum n_{a_1a_2} (\text{id}\circ \epsilon\otimes \text{id})^{a_1a_2}_{b_1b_2b_3} (R\otimes \text{id})^{b_1b_2b_3}_{c_1c_2c_3} (R\otimes \text{id})^{c_1c_2c_3}_{d_1d_2d_3} (\eta\otimes u')^{d_1d_2d_3}\\
    &= \sum n_{a_1a_2} \delta_{a_1b_1} \epsilon_{b_2} \delta_{a_2b_3} R^{b_1b_2}_{c_1c_2} \delta_{b_3c_3} R^{c_1c_2}_{d_1d_2} \delta_{c_3d_3} \eta^{d_1} u'^{d_2d_3}\\
    &= \sum n_{b_1b_3} \epsilon_{b_2} R^{b_1b_2}_{c_1c_2} R^{c_1c_2}_{d_1d_2} \eta^{d_1} u'^{d_2b_3}\\
    &= \sum_S \prod_E w(E),
\end{align*}
where the last equality follows from the change in indices
\[
    (a,b,c,d,e,f,g)\leftrightarrow(b_2,c_1,d_2,b_3,b_1,c_2,d_1).
\]
\end{example}

We conclude this subsection with some results on the behaviour of $Q^{A;\star}$ under taking the reverse $-K$ of a knotoid $K$; recall Definition \ref{def:reversal}.

\begin{lemma}\label{lm:reversal}
Let $K$ be an oriented biframed planar knotoid. Then $Q^{A;\star}(-K) = S\left(Q^{A;\star}(K)\right)$, where $S$ is the antipode of $A$.
\end{lemma}

This behaviour of $Q^{A;\star}$ is well-known \cite[Thm.~47]{bar2021}. For a proof see e.g.~\cite{habiro2006bottom}. We note the following result as a corollary:

\begin{corollary}\label{cor:S2}
Let $K$ be an oriented biframed planar knotoid. Then $Q^{A;\star}(K)$ is fixed under $S^2$, namely $S^2\left(Q^{A;\star}(K)\right)=Q^{A;\star}(K)$. Similarly $S^{-2}\left(Q^{A;\star}(K)\right)=Q^{A;\star}(K)$.
\end{corollary}

\begin{proof}
By Lemma \ref{lm:reversal} we have that $S^2\left(Q^{A;\star}(K)\right)=Q^{A;\star}(-(-K))$. From Figure \ref{fig:hooks} it is immediate that $-(-K)$ is just $K$ with a $+1$ coframing loop at its leg and a $-1$ coframing loop at its head. Thus by Lemma \ref{lm:diagrambijection} we have that $-(-K)\simeq K$. Therefore the first claim follows since $Q^{A;\star}$ is a biframed knotoid invariant. The second claim follows immediately from the first.
\end{proof}

\subsection{The Ribbon Hopf Algebra $\mathbb{D}$}


Having covered the generalities of quantum invariants of knotoids, we will now introduce a specific ribbon Hopf algebra $\mathbb{D}$ and discuss its associated universal invariant. These are both discussed in \cite{bar2021}, where a Mathematica implementation for computing $Q^{\mathbb{D};\star}$ is also given. After this subsection we use this implementation to compute example values of $Q^{\mathbb{D};\star}$ for several knotoids; also see Appendix \ref{app:code}.

In our discussion only the algebra structure, universal $R$-matrix, and the ribbon element will be needed. 

\begin{definition}
$\mathbb{D}$ is the algebra over $\mathbb{Q}[\epsilon]\llbracket \hbar\rrbracket$ 
generated by $\mathbf{y},\mathbf{b},\mathbf{a},\mathbf{x}$ subject to the relations
\[
    \mathbf{x}\mathbf{y} = e^{\epsilon\hbar} \mathbf{y}\mathbf{x} + \frac{1-e^{-\epsilon\hbar\mathbf{a}-\hbar\mathbf{b}}}{\hbar}
\]
\[
    [\mathbf{a},\mathbf{x}] = \mathbf{x},\quad
    [\mathbf{b},\mathbf{x}] = \epsilon\mathbf{x},\quad
    [\mathbf{a},\mathbf{y}] = -\mathbf{y},\quad
    [\mathbf{b},\mathbf{y}] = -\epsilon\mathbf{y},\quad
    [\mathbf{a},\mathbf{b}] = 0.
\]
\end{definition}

Before we can introduce the $R$-matrix we set $q=e^{\epsilon\hbar}$ and
$[k]_q! = \prod_{j=1}^k\frac{1-q^k}{1-q}$. Using the Drinfeld double construction it was found in \cite{bar2021} that:

\begin{theorem}
$\mathbb{D}$ a ribbon Hopf algebra with quasitriangular structure
\[
    \mathcal{R} = \sum_{m,n=0}^\infty \frac{\hbar^{m+n}}{[m]_q!n!} \mathbf{y}^m\mathbf{b}^n\otimes \mathbf{a}^n\mathbf{x}^m
\]
and with ribbon element $v$ yielding
\[
    uv^{-1} = \left( e^{-\epsilon\hbar\mathbf{a}-\hbar\mathbf{b}} \right)^{\frac{1}{2}}.
\]
\end{theorem}

Since $\mathbb{D}$ is ribbon, the techniques of Section \ref{subsec:universal_invariants} yield a universal quantum invariant of biframed knotoids associated to $\mathbb{D}$.

\begin{remark}
To discuss the implementation from \cite{bar2021} for computing $Q^{\mathbb{D};\star}$ we will use notation from \cite{bar2021}. Accordingly, we let $Z_\mathbb{D}$ denote the universal quantum invariant $Q^{\mathbb{D};\star}$ associated to $\mathbb{D}$. We will also draw our crossings and endpoints directed upwards, rather than downwards, from here on.
\end{remark}

To get an idea of the strength of $Z_\mathbb{D}$ as an invariant of biframed knotoids, we note that $Z_\mathbb{D}$ is stronger than all the colored Jones polynomials of biframed knotoids, where the $k$-th colored Jones polynomial is the Reshetikhin-Turaev invariant associated to $\left( U_\hbar(\mathfrak{sl}_2) , V_k \right)$. Here $U_\hbar(\mathfrak{sl}_2)$ is the well-known quantum group given by a deformation of the universal enveloping algebra of $\mathfrak{sl}_2$, and $V_k$ is its $k$-th irreducible representation. To see that $Z_\mathbb{D}$ is indeed stronger than the colored Jones polynomials, it suffices to note that $U_\hbar(\mathfrak{sl}_2)$ is isomorphic to a quotient of $\mathbb{D}$. Indeed: it is shown in \cite{bar2021} that setting $\epsilon=1$ and $\mathbf{t}\defeq \mathbf{b}-\epsilon\mathbf{a}=0$ in $\mathbb{D}$ yields a ribbon Hopf algebra isomorphic to $U_\hbar(\mathfrak{sl}_2)$.

\subsection{Implementation}\label{subsec:implementation}

We will briefly describe how the computer program given in \cite[App.~B]{bar2021} can be used to compute the invariant $Z_\mathbb{D}$ for knotoids. The Mathematica implementation can also be found on the second author's \href{http://rolandvdv.nl/PG/}{website}\footnote{\texttt{http://rolandvdv.nl/PG/}}. It carries out computations in $\mathbb{D}$ up to fixed order $k$ in $\epsilon$, with $k=1$ being sufficient for our purposes.

Say we are given a biframed knotoid diagram $K$. For simplicity we assume the diagram is upright in that all crossings point upwards. This means the cups and caps come in pairs that either don't contribute to $Z_\mathbb{D}$, or that form a full rotation which we call $C$ when it is counter-clockwise and $C^{-1}$ if it is clockwise. Now assign a label to each underpass and overpass of the crossings and also a label to each $C^{\pm}$, making sure no labels appear twice. Denote by $R_{ij}^\pm$ the positive/negative crossing with upper strand labeled $i$ and lower strand labeled $j$. Also denote by $C^{\pm}_i$ any edge rotating (counter)clockwise carrying label $i$. The knotoid diagram is the result of connecting each crossing and copy of $C$ in the right order, using strands that don't contribute to $Z_\mathbb{D}$. We call such a presentation for a knotoid diagram a \textbf{rotational tangle decomposition}.

In the program the edges with labels $u,v$ are connected by writing $m_{u,v\to w}$. This produces a longer strand that is now labeled $w$. Repeating this process for all components $R_{ij}^\pm$, $C^{\pm}_i$ in the knotoid diagram yields an expression for $Z_\mathbb{D}(K)$ in terms of a multiplication of copies of the values of $R_{ij}^\pm$ and $C^{\pm}_i$. The implementation uses efficient expressions of these values as well as of the multiplication $m_{u,v\to w}$ (namely in terms of `perturbed Gaussian generating functions'; see \cite{bar2021}) to evaluate this expression in $\mathbb{D}$.

A use case for the Mathematica implementation is given in Example \ref{ex:mathematica} below.






\begin{example}\label{ex:mathematica}
Say we consider the oriented biframed planar knotoid diagram $K$ depicted in Figure \ref{fig:5_7}, which is a diagram for the prime planar knotoid tabulated as $5_7$ in \cite{goundaroulis2019systematic}.

\begin{figure}[h]
    \centering
    \includegraphics[width=.5\linewidth]{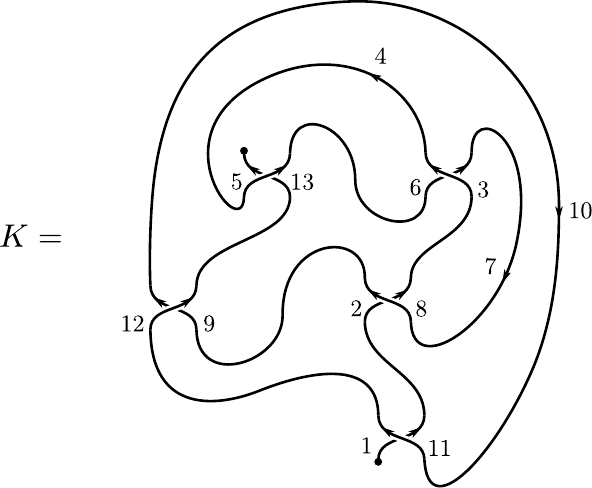}
    \caption{An oriented biframed planar knotoid diagram $K$ representing $5_7$, with labels giving a rotational tangle decomposition of $K$.}
    \label{fig:5_7}
\end{figure}

The diagram $K$ is drawn so that a rotational tangle decomposition of $K$ is immediately visible. This is indicated by the labels in Figure \ref{fig:5_7}, representing the labels $i,j$ of copies of $R_{ij}^\pm$ and $C^{\pm}_i$ in the diagram. Given this rotational tangle decomposition for the diagram $K$ of $5_7$ we can easily read off an expression for $Z_\mathbb{D}$ and enter it into Mathematica as follows:

\begin{mmaCell}[morefunctionlocal={j}]{Input}
Z57 = \mmaSub{\mmaOver{tR}{_}}{11,1} \mmaSub{tR}{12,9} \
\mmaSub{\mmaOver{tR}{_}}{8,2} \mmaSub{\mmaOver{tR}{_}}{3,6} \
\mmaSub{tR}{5,13} \mmaSub{\mmaOver{tC}{_}}{10} \
\mmaSub{\mmaOver{tC}{_}}{7} \
\mmaSub{tC}{4};
Do[Z57 = Z57 // \mmaSub{tm}{1,j\(\pmb{\to}\)1},\{j,2,13\}];
PowerExpand[Z57[[3]] // Simplify]
\end{mmaCell}

Here the first line feeds in the expression for $Z_\mathbb{D}(K)$. The second line carries out the multiplication of the elements of $\mathbb{D}$ associated to the rotational tangle components listed in the first line. The third line simplifies the resulting element of $\mathbb{D}$ to a manageable expression.

Running this code will compute $Z_\mathbb{D}(K)$ up to the selected order $k$ in $\epsilon$. For our purposes it will suffice to take $k=1$. The output is then as follows:

\begin{mmaCell}{Output}
\mmaSqrt{\mmaSub{T}{1}}+\mmaFrac{-4\mmaSubSup{a}{1}{2}\mmaSubSup{T}{1}{2}+\mmaSub{x}{1}\mmaSub{y}{1}(4+8\mmaSub{T}{1}+4\
\mmaSubSup{T}{1}{2}-3\mmaSub{x}{1}\mmaSub{y}{1})-4\mmaSub{a}{1}\mmaSub{T}{1}(-1+\mmaSub{T}{1}+\mmaSubSup{T}{1}{2}+2\mmaSub{x}{1}\mmaSub{y}{1})}{4\mmaSubSup{T}{1}{3/2}}\(\epsilon\)+\mmaSup{O[\(\epsilon\)]}{2}
\end{mmaCell}

The invariant we will be particularly interested in is the first-order coefficient of $Z_\mathbb{D}(K)$ in $\epsilon$. To obtain this invariant immediately we can alternatively run the following code:

\begin{mmaCell}[morefunctionlocal={j}]{Input}
Z57 = \mmaSub{\mmaOver{tR}{_}}{11,1} \mmaSub{tR}{12,9} \
\mmaSub{\mmaOver{tR}{_}}{8,2} \mmaSub{\mmaOver{tR}{_}}{3,6} \
\mmaSub{tR}{5,13} \mmaSub{\mmaOver{tC}{_}}{10} \
\mmaSub{\mmaOver{tC}{_}}{7} \
\mmaSub{tC}{4};
Do[Z57 = Z57 // \mmaSub{tm}{1,j\(\pmb{\to}\)1},\{j,2,13\}];
Coefficient[PowerExpand[Z57[[3]] // Simplify], \(\epsilon\), 1]
\end{mmaCell}

\begin{mmaCell}{Output}
\mmaFrac{-4\mmaSubSup{a}{1}{2}\mmaSubSup{T}{1}{2}+\mmaSub{x}{1}\mmaSub{y}{1}(4+8\mmaSub{T}{1}+4\mmaSubSup{T}{1}{2}-3\mmaSub{x}{1}\mmaSub{y}{1})-4\mmaSub{a}{1}\mmaSub{T}{1}(-1+\mmaSub{T}{1}+\mmaSubSup{T}{1}{2}+2\mmaSub{x}{1}\mmaSub{y}{1})}{4\mmaSubSup{T}{1}{3/2}}
\end{mmaCell}

Now let $-K$ be the reverse biframed diagram of the diagram in Figure \ref{fig:5_7}. By Lemma \ref{lm:reversal}, to compute $Z_\mathbb{D}(-K)$ we can apply $S$ to the result of computing $Z_\mathbb{D}(K)$. In Mathematica this is done as follows:

\begin{mmaCell}{Input}
Do[Z57op = Z57 // \mmaSub{tS}{1}, 1];
Coefficient[PowerExpand[Z57op[[3]] // Simplify], \(\epsilon\), 1]
\end{mmaCell}

\begin{mmaCell}{Output}
-\mmaFrac{4(2+\mmaSub{a}{1})\mmaSubSup{T}{1}{3}+4\mmaSubSup{T}{1}{2}(\mmaSub{a}{1}+\mmaSubSup{a}{1}{2}-\mmaSub{x}{1}\mmaSub{y}{1})+\mmaSub{x}{1}\mmaSub{y}{1}(-4+3\mmaSub{x}{1}\mmaSub{y}{1})+\mmaSub{T}{1}(-8(1+\mmaSub{x}{1}\mmaSub{y}{1})+\mmaSub{a}{1}(-4+8\mmaSub{x}{1}\mmaSub{y}{1}))}{4\mmaSubSup{T}{1}{3/2}}
\end{mmaCell}
\end{example}

Finally, we have mentioned earlier that this Mathematica implementation for computing $Z_\mathbb{D}$ is \textit{efficient}. Namely, the following is immediate from \cite[Thm.~50]{bar2021}:

\begin{corollary}
For $K$ a biframed planar knotoid diagram, one can compute $Z_\mathbb{D}(K)$  up to order $k$ in $\epsilon$ within $\mathcal{O}(n^{2k+2}\log(n))$ integer operations.
\end{corollary}

We conclude orders of $\epsilon$ in $Z_\mathbb{D}$ are computable in polynomial time for all knotoids, and the Mathematica implementation realizes this \cite{bar2021}.



\section{Examples}\label{sec:examples}

In this section we will refer to specific planar knotoids by using the labels with which they appear in the planar knotoid table given in \cite{goundaroulis2019systematic}. For example, $5_7$ refers to the $7$-th prime knotoid with $5$ crossings listed in that table. Table \ref{tb:unresolved_pairs} lists the pairs $(K_1,K_2)$ of prime planar knotoids with $5$ crossings for which it is conjectured in \cite{goundaroulis2019systematic} that $K_1\not\simeq K_2$.

\begin{table}[h]\label{tb:unresolved_pairs}
\centering
\begin{tabular}{l|l|l}
$K_1$    & $K_2$     & Oriented Gauss code                \\ \hline
$5_7$    & $5_{421}$ & -1 -2 3 4 -3 2 -5 1 5 -4 - - - + + \\
$5_9$    & $5_{561}$ & -1 2 -3 1 -4 5 -2 3 4 -5 - - - + + \\
$5_{12}$ & $5_{593}$ & -1 2 -3 1 4 -5 -2 3 -4 5 - - - - - \\
$5_{19}$ & $5_{796}$ & -1 2 -3 4 -5 1 -2 3 5 -4 - - - + + \\
$5_{21}$ & $5_{814}$ & -1 2 -3 4 -5 1 5 -2 -4 3 - + - - + \\
$5_{24}$ & $5_{891}$ & -1 2 -3 4 5 -4 -2 1 3 -5 - - - - +
\end{tabular}
\caption{Unresolved pairs $(K_1,K_2)$ of $5$-crossing prime planar knotoids from \cite{goundaroulis2019systematic} and their oriented Gauss codes.}
\end{table}

If all of these pairs were resolved, i.e.~shown to be equivalent or distinguished by some invariant of planar knotoids, the classification of prime planar knotoids with up to $5$ crossings would be complete. These pairs are particularly difficult to distinguish as they are all equivalent as spherical knotoids (recall any planar knotoid defines a spherical knotoid by adding $\{\infty\}$ to $\mathbb{R}^2$). This follows because they have diagrams with the same oriented Gauss codes, listed in Table \ref{tb:unresolved_pairs}; see \cite{goundaroulis2019systematic} for details. We claim that $Z_\mathbb{D}$ can distinguish several of these pairs. More specifically:

\begin{theorem}\label{thm:resolution}
The universal quantum invariant $Z_\mathbb{D}$ of oriented biframed planar knotoids can be used to distinguish the pairs $(5_9,5_{561})$ and $(5_{12},5_{593})$ of planar knotoids.
\end{theorem}

\begin{proof}
We will prove Theorem \ref{thm:resolution} by presenting the pairs $(K_1,K_2)$ by oriented biframed knotoid diagrams $D_1$ and $D_2$ with equal biframing. We will let $-D_1$ denote the diagram resulting from reversing the orientation of $D_1$.
We will then use the computer implementation of $Z_\mathbb{D}$ from Section \ref{subsec:implementation} to compute $Z_\mathbb{D}(D_1)$, $Z_\mathbb{D}(-D_1)$, and $Z_\mathbb{D}(D_2)$ up to first order in $\epsilon$. We shall see that $Z_\mathbb{D}(D_1)\neq Z_\mathbb{D}(D_2)$ and $Z_\mathbb{D}(-D_1)\neq Z_\mathbb{D}(D_2)$. Hence $D_1\not\simeq D_2$ as unoriented biframed knotoids. Since $D_1$ and $D_2$ were chosen to have equal biframing and $-D_1$ also has biframing equal to that of $D_1$, this implies $K_1\not\simeq K_2$ by Lemma \ref{lm:diagrambijection}.

We will describe the pairs $(D_1,D_2)$ and compute the associated values of $Z_\mathbb{D}$ at the end of this section and in Appendix \ref{app:code}.
\end{proof}

Since both pairs of biframed planar knotoids in Theorem \ref{thm:resolution} are distinguished by $Z_\mathbb{D}$ but equivalent as spherical knotoids, we conclude the following from Theorem \ref{thm:resolution}:

\begin{corollary}
$Z_\mathbb{D}=Q^{\mathbb{D};\star}$ is not an invariant of biframed spherical knotoids, which were defined in \cite{moltmaker}.
\end{corollary}

However, $Q^{A;\star}$ is an invariant of spherical knotoids under certain conditions on the ribbon Hopf algebra $A$:

\begin{lemma}
If $uv^{-1}\in A$ is self-inverse, i.e.~$(uv^{-1})^2=1\in A$, then $Q^{A;\star}$ is an invariant of spherical biframed knotoids. If $uv^{-1}=1$ then $Q^{A;\star}$ is in invariant of framed spherical knotoids, which were also defined in \cite{moltmaker}.
\end{lemma}

\begin{proof}
The first statement is immediate, since two equivalent biframed spherical knotoids can be related by a sequence of moves of biframed planar knotoids and moves replacing a rotational tangle component $C$ by $\overline{C}$ or vice versa. See \cite{moltmaker} for details. The second statement follows from noting that $1$ is obviously self-inverse, and that we can adjust the coframing of a biframed knotoid diagram $K$ by adding a number of rotational tangle components $C$ or $\overline{C}$ at either endpoint. Since $Q^{A;\star}$ associates $1\in A$ to these components, $Q^{A;\star}(K)$ is clearly independent of the chosen coframing of $K$.
\end{proof}

Note that an example where $uv^{-1}=1$ was already examined in \cite{moltmaker}.\\


With these corollaries out of the way, we go on to finish the proof of Theorem \ref{thm:resolution} by giving the appropriate diagrams for $5_9$, $5_{12}$, $5_{561}$ and $5_{593}$, and computing parts of $Z_\mathbb{D}$ for these diagrams. 

\begin{proof}
(\textit{Rest of the proof of Theorem \ref{thm:resolution}})
Oriented diagrams for $5_9$ and $5_{561}$ are given in Figure \ref{fig:9and561}. These diagrams are drawn to have equal biframing, as required by the proof of Theorem \ref{thm:resolution}. Note that these diagrams for $5_9$ and $5_{561}$ can be related to those given in \cite{goundaroulis2019systematic} by planar isotopy, as can be checked using the extended Gauss codes of the diagrams, for example.

\begin{figure}[h]
    \centering
    \includegraphics[width=.7\linewidth]{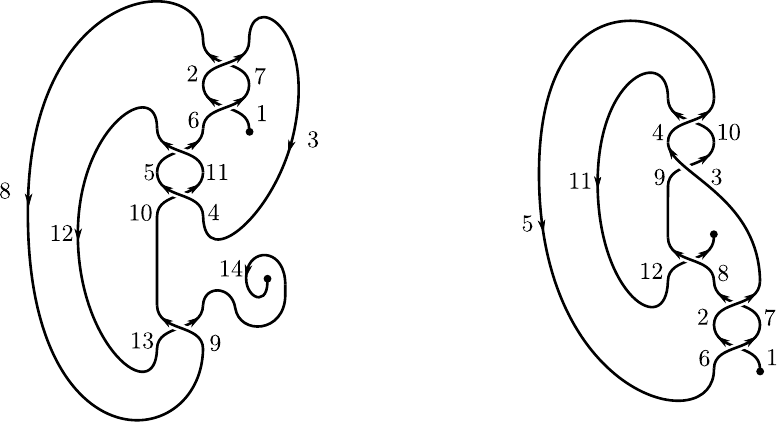}
    \caption{Biframed planar knotoid diagrams for $5_9$ (left) and $5_{561}$ (right).}
    \label{fig:9and561}
\end{figure}

To see that the diagrams in Figure \ref{fig:9and561} indeed have equal biframings, one can compute the writhe of both diagrams to conclude they both have framing $-1$, and use Remark \ref{rk:coframing_comp} to find both have coframing $-1$. Diagrams for $5_{12}$ and $5_{593}$ can be obtained by replacing the positive crossings in Figure \ref{fig:9and561} by negative ones; see Figure \ref{fig:12and593}. These diagrams both have biframing $(-5,-1)$ as can be checked by the methods mentioned above.

\begin{figure}[h]
    \centering
    \includegraphics[width=.7\linewidth]{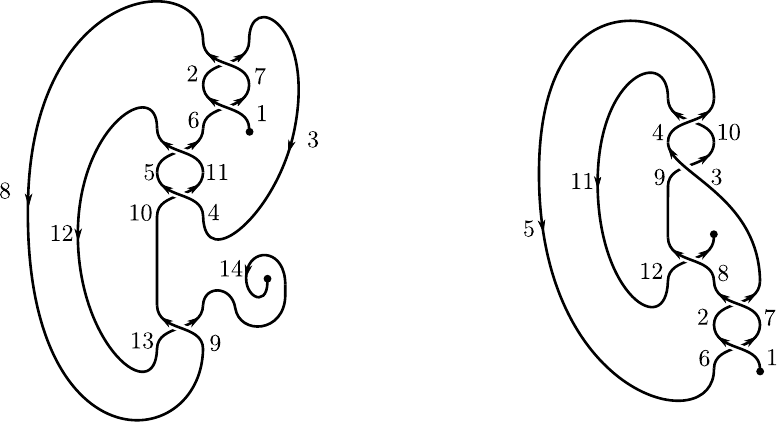}
    \caption{Biframed planar knotoid diagrams for $5_{12}$ (left) and $5_{593}$.}
    \label{fig:12and593}
\end{figure}

Now that we have diagrams for all the knotoids under consideration we would like to compute $Z_\mathbb{D}$ for these diagrams. To do so we have purposefully drawn them so that their rotational tangle decompositions can be read of immediately, as indicated by the labels in Figures \ref{fig:9and561} and \ref{fig:12and593}. We therefore find the following decompositions for these diagrams:
\begin{align*}
    & 5_{9} =  R_{6,1}\, R_{2,7}\, \overline{C}_{3}\, \overline{R}_{11,5}\, \overline{R}_{4,10}\, C_{8}\, \overline{R}_{9,13}\, C_{12}\, C_{14}, \\
    & 5_{12} = \overline{R}_{6,1}\, \overline{R}_{2,7}\, \overline{C}_{3}\, \overline{R}_{11,5}\, \overline{R}_{4,10}\, C_{8}\, \overline{R}_{9,13}\, C_{12}\, C_{14}, \\
    & 5_{561} = R_{6,1}\, R_{2,7}\, \overline{R}_{3,9}\, \overline{R}_{10,4}\, C_5\, \overline{R}_{8,12}\, C_{11}, \\
    & 5_{593} = \overline{R}_{6,1}\, \overline{R}_{2,7}\, \overline{R}_{3,9}\, \overline{R}_{10,4}\, C_5\, \overline{R}_{8,12}\, C_{11}.
\end{align*}
By way of a slight abuse of notation, in these decompositions we have used $5_{9}$ to denote the diagram for $5_{9}$ in Figure \ref{fig:9and561}, and similarly for the other diagrams. 

Armed with the decompositions above, it is a simple matter of rerunning the steps of Example \ref{ex:mathematica} to perform the computations of $Z_\mathbb{D}$ necessary to finish the proof of Theorem \ref{thm:resolution}. These computations are given in Appendix \ref{app:code}, where the first orders in $\epsilon$ are given for $Z_\mathbb{D}(D_1)$, $Z_\mathbb{D}(-D_1)$, and $Z_\mathbb{D}(D_2)$. Here $(D_1,D_2)$ are either $(5_9,5_{561})$ or $(5_{12},5_{593})$, as described in the proof of Theorem \ref{thm:resolution}. The results in Appendix \ref{app:code} show that in both cases, $Z_\mathbb{D}(D_1)\neq Z_\mathbb{D}(D_2)$ and $Z_\mathbb{D}(-D_1)\neq Z_\mathbb{D}(D_2)$ up to first order in $\epsilon$, meaning that $Z_\mathbb{D}(D_1)\neq Z_\mathbb{D}(D_2)$ and $Z_\mathbb{D}(-D_1)\neq Z_\mathbb{D}(D_2)$ as required. This finishes the proof of Theorem \ref{thm:resolution}.
\end{proof}

\textbf{Acknowledgements:} The first author would like to thank Jo Ellis-Monaghan for her guidance, and Kenzo Yasaka and Daniel Boutros for several helpful conversations.



\appendix

\section{Mathematica Computations}\label{app:code}

Below are the computations needed for the proof of Theorem \ref{thm:resolution}, implemented in Mathematica using the code provided in \cite[App.~B]{bar2021}. The implementation can also be found on the second author's website \href{http://rolandvdv.nl/PG/}{\texttt{http://rolandvdv.nl/PG/}}. In the results of these computations we use the notation $T=e^{-\hbar(\mathbf{b}-\epsilon \mathbf{a})}$. For completeness, our calculations on $(5_7,5_{421})$ are also included. Note that our calculations show $Z_{\mathbb{D}}(5_7)=Z_{\mathbb{D}}(5_{421})$ up to first order in $\epsilon$, meaning we \textbf{cannot} distinguish this pair using the methods described in Section \ref{sec:examples}.\\

\noindent\textbf{$5_7$ and $5_{421}$:}

\begin{mmaCell}[morefunctionlocal={j}]{Input}
Z57 = \mmaSub{\mmaOver{tR}{_}}{11,1} \mmaSub{tR}{12,9} \
\mmaSub{\mmaOver{tR}{_}}{8,2} \mmaSub{\mmaOver{tR}{_}}{3,6} \
\mmaSub{tR}{5,13} \mmaSub{\mmaOver{tC}{_}}{10} \
\mmaSub{\mmaOver{tC}{_}}{7} \
\mmaSub{tC}{4};
Do[Z57 = Z57 // \mmaSub{tm}{1,j\(\pmb{\to}\)1},\{j,2,13\}];
Coefficient[PowerExpand[Z57[[3]] // Simplify], \(\epsilon\), 1]
\end{mmaCell}

\begin{mmaCell}{Output}
\mmaFrac{-4\mmaSubSup{a}{1}{2}\mmaSubSup{T}{1}{2}+\mmaSub{x}{1}\mmaSub{y}{1}(4+8\mmaSub{T}{1}+4\mmaSubSup{T}{1}{2}-3\mmaSub{x}{1}\mmaSub{y}{1})-4\mmaSub{a}{1}\mmaSub{T}{1}(-1+\mmaSub{T}{1}+\mmaSubSup{T}{1}{2}+2\mmaSub{x}{1}\mmaSub{y}{1})}{4\mmaSubSup{T}{1}{3/2}}
\end{mmaCell}

\begin{mmaCell}{Input}
Do[Z57op = Z57 // \mmaSub{tS}{1}, 1];
Coefficient[PowerExpand[Z57op[[3]] // Simplify], \(\epsilon\), 1]
\end{mmaCell}

\begin{mmaCell}{Output}
-\mmaFrac{4(2+\mmaSub{a}{1})\mmaSubSup{T}{1}{3}+4\mmaSubSup{T}{1}{2}(\mmaSub{a}{1}+\mmaSubSup{a}{1}{2}-\mmaSub{x}{1}\mmaSub{y}{1})+\mmaSub{x}{1}\mmaSub{y}{1}(-4+3\mmaSub{x}{1}\mmaSub{y}{1})+\mmaSub{T}{1}(-8(1+\mmaSub{x}{1}\mmaSub{y}{1})+\mmaSub{a}{1}(-4+8\mmaSub{x}{1}\mmaSub{y}{1}))}{4\mmaSubSup{T}{1}{3/2}}
\end{mmaCell}

\begin{mmaCell}[morefunctionlocal={j}]{Input}
Z5421 = \mmaSub{tR}{9,1} \mmaSub{tR}{2,5} \mmaSub{\mmaOver{tR}{_}}{3,13} \
\mmaSub{\mmaOver{tR}{_}}{6,12} \mmaSub{\mmaOver{tR}{_}}{10,7} \
\mmaSub{tC}{11} \mmaSub{tC}{8} \mmaSub{\mmaOver{tC}{_}}{4};
Do[Z5421 = Z5421 // \mmaSub{tm}{1,j\(\pmb{\to}\)1},\{j,2,13\}];
Coefficient[PowerExpand[Z5421[[3]] // Simplify], \(\epsilon\), 1]
\end{mmaCell}

\begin{mmaCell}{Output}
\mmaFrac{-4\mmaSubSup{a}{1}{2}\mmaSubSup{T}{1}{2}+\mmaSub{x}{1}\mmaSub{y}{1}(4+8\mmaSub{T}{1}+4\mmaSubSup{T}{1}{2}-3\mmaSub{x}{1}\mmaSub{y}{1})-4\mmaSub{a}{1}\mmaSub{T}{1}(-1+\mmaSub{T}{1}+\mmaSubSup{T}{1}{2}+2\mmaSub{x}{1}\mmaSub{y}{1})}{4\mmaSubSup{T}{1}{3/2}}
\end{mmaCell}

\noindent\textbf{$5_9$ and $5_{561}$:}

\begin{mmaCell}[morefunctionlocal={j}]{Input}
Z59 = \mmaSub{tR}{6,1} \mmaSub{tR}{2,7} \mmaSub{\mmaOver{tR}{_}}{11,5} \
\mmaSub{\mmaOver{tR}{_}}{4,10} \mmaSub{\mmaOver{tR}{_}}{9,13} \
\mmaSub{\mmaOver{tC}{_}}{3} \mmaSub{tC}{12} \mmaSub{tC}{8} \
\mmaSub{tC}{14};
Do[Z59 = Z59 // \mmaSub{tm}{1,j\(\pmb{\to}\)1},\{j,2,?\}];
Coefficient[PowerExpand[Z59[[3]] // Simplify], \(\epsilon\), 1]
\end{mmaCell}

\begin{mmaCell}{Output}
4+2\mmaSub{a}{1}+\mmaSub{x}{1}\mmaSub{y}{1}-\mmaSubSup{x}{1}{2}\mmaSubSup{y}{1}{2}-\mmaFrac{\mmaSub{x}{1}\mmaSub{y}{1}(1+\mmaSub{x}{1}\mmaSub{y}{1})}{\mmaSubSup{T}{1}{2}}+\mmaSubSup{T}{1}{2}(2+\mmaSub{x}{1}\mmaSub{y}{1}+\mmaSub{a}{1}(1+\mmaSub{x}{1}\mmaSub{y}{1}))
  -\mmaFrac{8+4\mmaSub{x}{1}\mmaSub{y}{1}-\mmaSubSup{x}{1}{2}\mmaSubSup{y}{1}{2}+4\mmaSub{a}{1}(1+\mmaSub{x}{1}\mmaSub{y}{1})}{4\mmaSub{T}{1}}-\mmaSub{T}{1}(4+\mmaSubSup{a}{1}{2}+3\mmaSub{x}{1}\mmaSub{y}{1}+2\mmaSub{a}{1}(2+\mmaSub{x}{1}\mmaSub{y}{1}))
\end{mmaCell}

\begin{mmaCell}{Input}
Do[Z59op = Z59 // \mmaSub{tS}{1}, 1];
Coefficient[PowerExpand[Z59op[[3]] // Simplify], \(\epsilon\), 1]
\end{mmaCell}

\begin{mmaCell}{Output}
-\mmaSubSup{a}{1}{2}+\mmaFrac{\mmaSub{a}{1}(-1+2\mmaSub{T}{1}-\mmaSub{x}{1}\mmaSub{y}{1}-2\mmaSubSup{T}{1}{2}(1+\mmaSub{x}{1}\mmaSub{y}{1})+\mmaSubSup{T}{1}{3}(1+\mmaSub{x}{1}\mmaSub{y}{1}))}{\mmaSubSup{T}{1}{2}}
  -\mmaFrac{\mmaSub{x}{1}\mmaSub{y}{1}(4-4\mmaSubSup{T}{1}{3}+4\mmaSubSup{T}{1}{4}+4\mmaSub{x}{1}\mmaSub{y}{1}+4\mmaSubSup{T}{1}{2}(1+\mmaSub{x}{1}\mmaSub{y}{1})-\mmaSub{T}{1}(4+\mmaSub{x}{1}\mmaSub{y}{1}))}{4\mmaSubSup{T}{1}{3}}
\end{mmaCell}

\begin{mmaCell}[morefunctionlocal={j}]{Input}
Z5561 = \mmaSub{tR}{6,1} \mmaSub{tR}{2,7} \mmaSub{\mmaOver{tR}{_}}{8,12} \
\mmaSub{\mmaOver{tR}{_}}{3,9} \mmaSub{\mmaOver{tR}{_}}{10,4} \
\mmaSub{tC}{11} \mmaSub{tC}{5};
Do[Z5561 = Z5561 // \mmaSub{tm}{1,j\(\pmb{\to}\)1},\{j,2,12\}];
Coefficient[PowerExpand[Z5561[[3]] // Simplify], \(\epsilon\), 1]
\end{mmaCell}

\begin{mmaCell}{Output}
-\mmaSubSup{a}{1}{2}\mmaSub{T}{1}+\mmaFrac{\mmaSub{a}{1}(-1+2\mmaSub{T}{1}-\mmaSub{x}{1}\mmaSub{y}{1}+\mmaSubSup{T}{1}{3}(1+\mmaSub{x}{1}\mmaSub{y}{1})-2\mmaSubSup{T}{1}{2}(2+\mmaSub{x}{1}\mmaSub{y}{1}))}{\mmaSub{T}{1}}
  -\mmaFrac{\mmaSub{x}{1}\mmaSub{y}{1}(4-4\mmaSubSup{T}{1}{3}+4\mmaSubSup{T}{1}{4}+4\mmaSub{x}{1}\mmaSub{y}{1}+4\mmaSubSup{T}{1}{2}(3+\mmaSub{x}{1}\mmaSub{y}{1})-\mmaSub{T}{1}(4+\mmaSub{x}{1}\mmaSub{y}{1}))}{4\mmaSubSup{T}{1}{2}}
\end{mmaCell}

\noindent\textbf{$5_{12}$ and $5_{593}$:}

\begin{mmaCell}[morefunctionlocal={j}]{Input}
Z512 = \mmaSub{\mmaOver{tR}{_}}{1,6} \mmaSub{\mmaOver{tR}{_}}{7,2} \
\mmaSub{\mmaOver{tR}{_}}{11,5} \mmaSub{\mmaOver{tR}{_}}{4,10} \
\mmaSub{\mmaOver{tR}{_}}{9,13} \mmaSub{\mmaOver{tC}{_}}{3} \
\mmaSub{tC}{12} \mmaSub{tC}{8} \mmaSub{tC}{14};
Do[Z512 = Z512 // \mmaSub{tm}{1,j\(\pmb{\to}\)1},\{j,2,14\}];
Coefficient[PowerExpand[Z512[[3]] // Simplify], \(\epsilon\), 1]
\end{mmaCell}

\begin{mmaCell}{Output}
\mmaFrac{1}{4\mmaSubSup{T}{1}{5}\mmaSup{(1-\mmaSub{T}{1}+\mmaSubSup{T}{1}{2}-\mmaSubSup{T}{1}{3}+\mmaSubSup{T}{1}{4})}{3}}(-4\mmaSub{a}{1}(2+5\mmaSub{a}{1})\mmaSubSup{T}{1}{18}-4(3+20\mmaSub{a}{1})\mmaSubSup{T}{1}{7}\mmaSub{x}{1}\mmaSub{y}{1}-120\mmaSub{a}{1}\mmaSubSup{T}{1}{9}\mmaSub{x}{1}\mmaSub{y}{1}
  -19\mmaSubSup{x}{1}{2}\mmaSubSup{y}{1}{2}+6\mmaSub{T}{1}\mmaSubSup{x}{1}{2}\mmaSubSup{y}{1}{2}-32\mmaSubSup{T}{1}{2}\mmaSubSup{x}{1}{2}\mmaSubSup{y}{1}{2}+18\mmaSubSup{T}{1}{3}\mmaSubSup{x}{1}{2}\mmaSubSup{y}{1}{2}-39\mmaSubSup{T}{1}{4}\mmaSubSup{x}{1}{2}\mmaSubSup{y}{1}{2}+\mmaSubSup{T}{1}{8}\mmaSub{x}{1}\mmaSub{y}{1}(-4+96\mmaSub{a}{1}-35\mmaSub{x}{1}\mmaSub{y}{1})
  +4\mmaSubSup{T}{1}{6}\mmaSub{x}{1}\mmaSub{y}{1}(1+12\mmaSub{a}{1}-11\mmaSub{x}{1}\mmaSub{y}{1})+2\mmaSubSup{T}{1}{5}\mmaSub{x}{1}\mmaSub{y}{1}(-8-20\mmaSub{a}{1}+\mmaSub{x}{1}\mmaSub{y}{1})+\mmaSubSup{T}{1}{11}(4+40\mmaSubSup{a}{1}{2}-20\mmaSub{x}{1}\mmaSub{y}{1}
  -6\mmaSubSup{x}{1}{2}\mmaSubSup{y}{1}{2}+\mmaSub{a}{1}(60-84\mmaSub{x}{1}\mmaSub{y}{1}))+2\mmaSubSup{T}{1}{13}(-8+40\mmaSubSup{a}{1}{2}-24\mmaSub{x}{1}\mmaSub{y}{1}-3\mmaSubSup{x}{1}{2}\mmaSubSup{y}{1}{2}+\mmaSub{a}{1}(44-24\mmaSub{x}{1}\mmaSub{y}{1}))
  -4\mmaSubSup{T}{1}{10}(2+5\mmaSubSup{a}{1}{2}+2\mmaSub{x}{1}\mmaSub{y}{1}+7\mmaSubSup{x}{1}{2}\mmaSubSup{y}{1}{2}+\mmaSub{a}{1}(9-17\mmaSub{x}{1}\mmaSub{y}{1}))+4\mmaSubSup{T}{1}{17}(-1+10\mmaSubSup{a}{1}{2}-2\mmaSub{x}{1}\mmaSub{y}{1}
  +\mmaSub{a}{1}(5-3\mmaSub{x}{1}\mmaSub{y}{1}))-4\mmaSubSup{T}{1}{14}(-6+25\mmaSubSup{a}{1}{2}-11\mmaSub{x}{1}\mmaSub{y}{1}+2\mmaSubSup{x}{1}{2}\mmaSubSup{y}{1}{2}-6\mmaSub{a}{1}(-4+\mmaSub{x}{1}\mmaSub{y}{1}))
  +\mmaSubSup{T}{1}{16}(16-60\mmaSubSup{a}{1}{2}+20\mmaSub{x}{1}\mmaSub{y}{1}-3\mmaSubSup{x}{1}{2}\mmaSubSup{y}{1}{2}+12\mmaSub{a}{1}(-3+\mmaSub{x}{1}\mmaSub{y}{1}))+\mmaSubSup{T}{1}{12}(8-60\mmaSubSup{a}{1}{2}+24\mmaSub{x}{1}\mmaSub{y}{1}-15\mmaSubSup{x}{1}{2}\mmaSubSup{y}{1}{2}
  +\mmaSub{a}{1}(-76+52\mmaSub{x}{1}\mmaSub{y}{1}))+\mmaSubSup{T}{1}{15}(80\mmaSubSup{a}{1}{2}+\mmaSub{a}{1}(60-36\mmaSub{x}{1}\mmaSub{y}{1})-2(12+18\mmaSub{x}{1}\mmaSub{y}{1}+\mmaSubSup{x}{1}{2}\mmaSubSup{y}{1}{2})))
\end{mmaCell}

\begin{mmaCell}{Input}
Do[Z512op = Z512 // \mmaSub{tS}{1}, 1];
Coefficient[PowerExpand[Z512op[[3]] // Simplify], \(\epsilon\), 1]
\end{mmaCell}

\begin{mmaCell}{Output}
\mmaFrac{1}{4\mmaSubSup{T}{1}{6}\mmaSup{(1-\mmaSub{T}{1}+\mmaSubSup{T}{1}{2}-\mmaSubSup{T}{1}{3}+\mmaSubSup{T}{1}{4})}{3}}(-20\mmaSubSup{a}{1}{2}\mmaSubSup{T}{1}{18}+4(1-20\mmaSub{a}{1})\mmaSubSup{T}{1}{7}\mmaSub{x}{1}\mmaSub{y}{1}-24(-1+5\mmaSub{a}{1})\mmaSubSup{T}{1}{9}\mmaSub{x}{1}\mmaSub{y}{1}
  -19\mmaSubSup{x}{1}{2}\mmaSubSup{y}{1}{2}+6\mmaSub{T}{1}\mmaSubSup{x}{1}{2}\mmaSubSup{y}{1}{2}-32\mmaSubSup{T}{1}{2}\mmaSubSup{x}{1}{2}\mmaSubSup{y}{1}{2}+18\mmaSubSup{T}{1}{3}\mmaSubSup{x}{1}{2}\mmaSubSup{y}{1}{2}-39\mmaSubSup{T}{1}{4}\mmaSubSup{x}{1}{2}\mmaSubSup{y}{1}{2}+2\mmaSubSup{T}{1}{5}\mmaSub{x}{1}\mmaSub{y}{1}(-4-20\mmaSub{a}{1}+\mmaSub{x}{1}\mmaSub{y}{1})
  -4\mmaSubSup{T}{1}{6}\mmaSub{x}{1}\mmaSub{y}{1}(1-12\mmaSub{a}{1}+11\mmaSub{x}{1}\mmaSub{y}{1})+\mmaSubSup{T}{1}{11}(28+40\mmaSubSup{a}{1}{2}+28\mmaSub{x}{1}\mmaSub{y}{1}-6\mmaSubSup{x}{1}{2}\mmaSubSup{y}{1}{2}+\mmaSub{a}{1}(44-84\mmaSub{x}{1}\mmaSub{y}{1}))
  -4\mmaSubSup{T}{1}{10}(4+5\mmaSubSup{a}{1}{2}+8\mmaSub{x}{1}\mmaSub{y}{1}+7\mmaSubSup{x}{1}{2}\mmaSubSup{y}{1}{2}+\mmaSub{a}{1}(7-17\mmaSub{x}{1}\mmaSub{y}{1}))-4\mmaSubSup{T}{1}{14}(6+25\mmaSubSup{a}{1}{2}+5\mmaSub{x}{1}\mmaSub{y}{1}+2\mmaSubSup{x}{1}{2}\mmaSubSup{y}{1}{2}
  +\mmaSub{a}{1}(14-6\mmaSub{x}{1}\mmaSub{y}{1}))+4\mmaSubSup{T}{1}{17}(1+10\mmaSubSup{a}{1}{2}+2\mmaSub{x}{1}\mmaSub{y}{1}+\mmaSub{a}{1}(1-3\mmaSub{x}{1}\mmaSub{y}{1}))-\mmaSubSup{T}{1}{12}(32+60\mmaSubSup{a}{1}{2}+32\mmaSub{x}{1}\mmaSub{y}{1}
  +15\mmaSubSup{x}{1}{2}\mmaSubSup{y}{1}{2}-52\mmaSub{a}{1}(-1+\mmaSub{x}{1}\mmaSub{y}{1}))-\mmaSubSup{T}{1}{16}(8+60\mmaSubSup{a}{1}{2}+12\mmaSub{x}{1}\mmaSub{y}{1}+3\mmaSubSup{x}{1}{2}\mmaSubSup{y}{1}{2}-12\mmaSub{a}{1}(-1+\mmaSub{x}{1}\mmaSub{y}{1}))
  +\mmaSubSup{T}{1}{13}(32+80\mmaSubSup{a}{1}{2}+24\mmaSub{x}{1}\mmaSub{y}{1}-6\mmaSubSup{x}{1}{2}\mmaSubSup{y}{1}{2}-8\mmaSub{a}{1}(-7+6\mmaSub{x}{1}\mmaSub{y}{1}))+\mmaSubSup{T}{1}{8}\mmaSub{x}{1}\mmaSub{y}{1}(96\mmaSub{a}{1}-5(4+7\mmaSub{x}{1}\mmaSub{y}{1}))
  +2\mmaSubSup{T}{1}{15}(8+40\mmaSubSup{a}{1}{2}+10\mmaSub{x}{1}\mmaSub{y}{1}-\mmaSubSup{x}{1}{2}\mmaSubSup{y}{1}{2}-2\mmaSub{a}{1}(-7+9\mmaSub{x}{1}\mmaSub{y}{1})))
\end{mmaCell}

\begin{mmaCell}[morefunctionlocal={j}]{Input}
Z5593 = \mmaSub{\mmaOver{tR}{_}}{1,6} \mmaSub{\mmaOver{tR}{_}}{7,2} \
\mmaSub{\mmaOver{tR}{_}}{8,12} \mmaSub{\mmaOver{tR}{_}}{3,9} \mmaSub{\
\mmaOver{tR}{_}}{10,4} \mmaSub{tC}{11} \mmaSub{tC}{5};
Do[Z5593 = Z5593 // \mmaSub{tm}{1,j\(\pmb{\to}\)1},\{j,2,12\}];
Coefficient[PowerExpand[Z5593[[3]] // Simplify], \(\epsilon\), 1]
\end{mmaCell}

\begin{mmaCell}{Output}
\mmaFrac{1}{4\mmaSubSup{T}{1}{5}\mmaSup{(1-\mmaSub{T}{1}+\mmaSubSup{T}{1}{2}-\mmaSubSup{T}{1}{3}+\mmaSubSup{T}{1}{4})}{3}}(-4\mmaSub{a}{1}(2+5\mmaSub{a}{1})\mmaSubSup{T}{1}{18}-4(3+20\mmaSub{a}{1})\mmaSubSup{T}{1}{7}\mmaSub{x}{1}\mmaSub{y}{1}-120\mmaSub{a}{1}\mmaSubSup{T}{1}{9}\mmaSub{x}{1}\mmaSub{y}{1}
  -19\mmaSubSup{x}{1}{2}\mmaSubSup{y}{1}{2}+6\mmaSub{T}{1}\mmaSubSup{x}{1}{2}\mmaSubSup{y}{1}{2}-32\mmaSubSup{T}{1}{2}\mmaSubSup{x}{1}{2}\mmaSubSup{y}{1}{2}+18\mmaSubSup{T}{1}{3}\mmaSubSup{x}{1}{2}\mmaSubSup{y}{1}{2}-39\mmaSubSup{T}{1}{4}\mmaSubSup{x}{1}{2}\mmaSubSup{y}{1}{2}+\mmaSubSup{T}{1}{8}\mmaSub{x}{1}\mmaSub{y}{1}(-4+96\mmaSub{a}{1}-35\mmaSub{x}{1}\mmaSub{y}{1})
  +4\mmaSubSup{T}{1}{6}\mmaSub{x}{1}\mmaSub{y}{1}(1+12\mmaSub{a}{1}-11\mmaSub{x}{1}\mmaSub{y}{1})+2\mmaSubSup{T}{1}{5}\mmaSub{x}{1}\mmaSub{y}{1}(-8-20\mmaSub{a}{1}+\mmaSub{x}{1}\mmaSub{y}{1})+\mmaSubSup{T}{1}{11}(28+40\mmaSubSup{a}{1}{2}+4\mmaSub{x}{1}\mmaSub{y}{1}
  -6\mmaSubSup{x}{1}{2}\mmaSubSup{y}{1}{2}+\mmaSub{a}{1}(60-84\mmaSub{x}{1}\mmaSub{y}{1}))-\mmaSubSup{T}{1}{12}(32+60\mmaSubSup{a}{1}{2}+16\mmaSub{x}{1}\mmaSub{y}{1}+15\mmaSubSup{x}{1}{2}\mmaSubSup{y}{1}{2}+\mmaSub{a}{1}(76-52\mmaSub{x}{1}\mmaSub{y}{1}))
  +\mmaSubSup{T}{1}{13}(32+80\mmaSubSup{a}{1}{2}-6\mmaSubSup{x}{1}{2}\mmaSubSup{y}{1}{2}+\mmaSub{a}{1}(88-48\mmaSub{x}{1}\mmaSub{y}{1}))-4\mmaSubSup{T}{1}{10}(4+5\mmaSubSup{a}{1}{2}+4\mmaSub{x}{1}\mmaSub{y}{1}+7\mmaSubSup{x}{1}{2}\mmaSubSup{y}{1}{2}
  +\mmaSub{a}{1}(9-17\mmaSub{x}{1}\mmaSub{y}{1}))+4\mmaSubSup{T}{1}{17}(1+10\mmaSubSup{a}{1}{2}+\mmaSub{a}{1}(5-3\mmaSub{x}{1}\mmaSub{y}{1}))-4\mmaSubSup{T}{1}{14}(6+25\mmaSubSup{a}{1}{2}+\mmaSub{x}{1}\mmaSub{y}{1}+2\mmaSubSup{x}{1}{2}\mmaSubSup{y}{1}{2}
  -6\mmaSub{a}{1}(-4+\mmaSub{x}{1}\mmaSub{y}{1}))-\mmaSubSup{T}{1}{16}(8+60\mmaSubSup{a}{1}{2}+4\mmaSub{x}{1}\mmaSub{y}{1}+3\mmaSubSup{x}{1}{2}\mmaSubSup{y}{1}{2}-12\mmaSub{a}{1}(-3+\mmaSub{x}{1}\mmaSub{y}{1}))
  +2\mmaSubSup{T}{1}{15}(8+40\mmaSubSup{a}{1}{2}+2\mmaSub{x}{1}\mmaSub{y}{1}-\mmaSubSup{x}{1}{2}\mmaSubSup{y}{1}{2}-6\mmaSub{a}{1}(-5+3\mmaSub{x}{1}\mmaSub{y}{1})))
\end{mmaCell}

It may be non-obvious from this output that $Z_\mathbb{D}(5_{12})\neq Z_\mathbb{D}(5_{593})$, so we let Mathematica verify this below:

\begin{mmaCell}[moredefined={Z5593, Z512}]{Input}
PowerExpand[Simplify[Z5593[[3]]-Z512[[3]]]]
\end{mmaCell}

\begin{mmaCell}{Output}
\mmaFrac{2\mmaSubSup{T}{1}{5}(-1+2\mmaSub{T}{1}-2\mmaSubSup{T}{1}{2}+\mmaSubSup{T}{1}{3})(1+\mmaSub{x}{1}\mmaSub{y}{1})\mmaUnd{\(\pmb{\epsilon}\)}}{\mmaSup{(1-\mmaSub{T}{1}+\mmaSubSup{T}{1}{2}-\mmaSubSup{T}{1}{3}+\mmaSubSup{T}{1}{4})}{2}}+\mmaSup{O[\(\pmb{\epsilon}\)]}{2}
\end{mmaCell}

\end{document}